\pdfoutput=1
\documentclass[preprint]{elsarticle}
\usepackage{amsmath,graphicx}
\graphicspath{{./img/}}

\usepackage{enumitem}

\usepackage{algorithm} 
\usepackage{algorithmic}

\usepackage{amssymb}
\usepackage{amsfonts}
\usepackage{amsthm}
\usepackage{booktabs}
\usepackage{url}
\usepackage{mathtools}
\usepackage{subfigure}

\usepackage{color}
\usepackage{xcolor}

\usepackage{soul}
\usepackage{todonotes}
\usepackage[utf8]{inputenc}
\usepackage[T1]{fontenc}

\newcommand{\eg}{{\it e.g.}}  
\newcommand{\ie}{{\it i.e.}}
\newcommand{\ea}{{\it et al.}}

\newcommand{\reals}{{\mathbb R}}

\newcommand{\naturals}{\mathbb{N}}
\newcommand{\expect}{{\mathbb E}}
\newcommand{\prob}{{\mathbb P}}

\newcommand{\argmin}{\operatornamewithlimits{argmin}}

\newcommand{\minimize}{\operatornamewithlimits{minimize}}

\newcommand{\subjto}{\operatornamewithlimits{\text{subject to}}}

\newcommand {\mtxnorm}[1]{\left|\mathopen{}\left|\mathopen{}\left|
        {#1} \right|\mathclose{}\right|\mathclose{}\right|}

\newtheorem{theorem}{Theorem}
\newtheorem{lemma}[theorem]{Lemma}
\newtheorem{proposition}[theorem]{Proposition}
\newtheorem{assumption}[theorem]{Assumption}
\newtheorem{definition}[theorem]{Definition}

\begin{document}
\begin{frontmatter}
\title{STRONG: Synchronous and asynchronous RObust Network localization, under Non-Gaussian noise}
\author{Cl\'audia Soares\corref{mycorrespondingauthor}}
\cortext[mycorrespondingauthor]{Corresponding author}
\ead{csoares@isr.tecnico.ulisboa.pt}

\author{Jo\~ao Gomes}
\ead{jpg@isr.tecnico.ulisboa.pt}


\address{Institute for Systems and Robotics, Laboratory for Robotics
  and Engineering Systems, IST Universidade de Lisboa, Portugal}

\begin{abstract}
  Real-world network applications must cope with failing nodes,
  malicious attacks, or nodes facing corrupted data --- data
  classified as outliers. Our work addresses these concerns in the
  scope of the sensor network localization problem where, despite the abundance of
  technical literature, prior research seldom considered outlier
  data. We propose robust, fast, and distributed network localization
  algorithms, resilient to high-power noise, but also precise under
  regular Gaussian noise. We use a Huber M-estimator, thus obtaining a
  robust (but nonconvex) optimization problem. We convexify and change
  the problem representation, to allow for distributed robust
  localization algorithms: a synchronous distributed method that has
  optimal convergence rate and an asynchronous one with proven
  convergence guarantees. A major highlight of our contribution lies
  on the fact that we pay no price for provable distributed
  computation neither in accuracy, nor in communication cost or
  convergence speed. Simulations showcase the superior performance of
  our algorithms, both in the presence of outliers and under regular
  Gaussian noise: our method exceeds the accuracy of alternative
  approaches, distributed and centralized, even under heavy additive
  and multiplicative outlier noise.
\end{abstract}

\begin{keyword}
  Distributed Localization algorithms \sep Robust estimation \sep Huber
  function \sep convex relaxation \sep nonconvex optimization \sep distributed
  iterative network localization \sep sensor networks
\end{keyword}

\end{frontmatter}



\section{Introduction}
\label{sec:intro}
Outliers can cause large errors in non-robust estimation algorithms
and, if other systems use wrong estimates as input, error propagation
can invalidate the engineered system's final purpose. Network
localization is a key component in many network-centric systems that
is prone to such drastic error propagation.  It might be taken for
granted in most sensor network applications, but in challenging
environments network localization is an open and very active research
field. We present a new approach addressing the presence of outliers,
not by eliminating them from the estimation process, but by weighting
them, so they can contribute to the solution, while mitigating the
outlier bias on the estimator. {\color{blue} This strategy is strongly
  rooted in Robust Statistics and Robust Statistical Signal
  Processing~\cite{zoubir2012robust}.}

\subsection{Related work}
\label{sec:related-work}

Focusing on more related work, several different approaches are available,
some performing semidefinite or second-order cone relaxations of the
original nonconvex problem like O\u{g}uz-Ekim
\ea~\cite{OguzGomesXavierOliveira2011} or Biswas
\ea~\cite{BiswasLiangTohYeWang2006}. These approaches do not scale
well, since the centralized semidefinite program (SDP) or second-order
cone program (SOCP) gets very large even for a small number of nodes.
In~O\u{g}uz-Ekim \ea\ the Majorization-Minimization (MM) framework was
used with quadratic cost functions to also derive centralized
approaches to the sensor network localization problem.  Other
approaches rely on multidimensional scaling, where the sensor network
localization problem is posed as a least squares problem, as in the
excellent work of Shang
\ea~\cite{ShangRumiZhangFromherz2004}. Unfortunately, multidimensional
scaling is unreliable in large-scale networks with sparse
connectivity.  Another class of relaxations focuses on the envelopes
of terms in the cost function, like { the work
  in~\cite{soares2014simple}, where we present three algorithms for
  network localization, one parallel and two asynchronous ones, all
  with proven convergence. While developed as a relaxation of a
  maximum-likelihood estimator, \cite{soares2014simple} assumes an
  additive Gaussian noise measurement model which, as we now show
  through numerical simulations, does not withstand outlier noise. To
  address a data model where range measurements are contaminated with
  Gaussian measurement noise plus outliers, here we cast the
  localization problem in a broader M-estimation framework and replace
  the outlier-amplifying quadratic loss considered
  in~\cite{soares2014simple} by the robust Huber function. A Huber
  M-estimator was already used for centralized target localization
  in~\cite{Sun2004}, where the Huber estimation was used in a
  bootstrapping scheme.  The robust problem is also nonconvex and a
  global optimum is very difficult to find. As
  in~\cite{soares2014simple}, we opt for a term-to-term relaxation,
  but unlike~\cite{soares2014simple} the new convex approximated
  robust problem includes explicit constraints that elicit different
  solution methods. To find a distributed algorithm with low
  computational complexity, we put forward a novel representation of
  the Huber function composed with a norm.  We start by developing an
  optimal first-order synchronous algorithm, similarly
  to~\cite{soares2014simple}, but for a different, robust problem. As
  network localization in many applications calls for unstructured,
  asynchronous time models, we also develop an asynchronous algorithm
  based on coordinate gradient descent, for which we provide proof of
  convergence to the optimal cost with probability one. This parallels
  the result in the unconstrained Gaussian-motivated problem
  in~\cite{soares2014simple}, but for the robust problem with
  constraints.}

{ Some existing references directly tackle the nonconvex
  maximum-likelihood problem. Due to nonconvexity of the problem, such
  iterative algorithms can be trapped in one of the multiple local
  maxima of the cost, and consequently they aspire to no more than
  finding a local optimum that may be highly sensitive to the
  algorithm initialization. Here we can point out several approaches,
  like Costa \textit{et al.}~\cite{CostaPatwariHero2006}, where the
  authors present a distributed solution inspired in multidimensional
  scaling, {\color{blue}Ersenghe~\cite{Erseghe2015} that presents a
    nonconvex ADMM, proven to converge to a stationary point of the ML
    cost}, and Soares \textit{et al.}~\cite{SoaresXavierGomes2014a},
  where the authors reformulate the problem to obtain a Lipschitz
  gradient cost, decoupling the problem via MM.  } {\color{blue} A
  common approach to initialize the nonconvex maximum-likelihood
  algorithms is to run a convex relaxation, and refine the result with
  the nonconvex method. This type of approach can be found
  in~\cite{piovesan2016}.}

Some of the localization fundamental limits were object of thorough
study, for example in~\cite{Bletsas2016}, the authors derive
closed-form Cram\'{e}r-Rao bounds for multi-modal, multi-agent
systems, with remarkably general assumptions, but only covering
independent Gaussian noise. The very recent work~\cite{Buehrer2018}
notes that the Cram\'{e}r-Rao lower bound considered in previous
research assumes a fixed configuration of anchoring nodes, thus not
conveying the full picture regarding localization performance; to
address this gap, the authors consider the bound as a random variable,
and put forward the distribution of the Cram\'{e}r-Rao lower bound
under a time of arrival measurement model. The statistical assumptions
on this result include, again, normally distributed independent
measurement noise. The recent rise of 5G and full-duplex wireless
communications opened the door to interesting developments in
GPS-denied localization technologies, as in~\cite{Shen2018},
and~\cite{DestinoWymeersch2018}, both presenting fundamental bounds
and estimators for position and orientation, once more assuming
normally distributed noise.

All these approaches assume Gaussian noise contaminating the distance
measurements or their squares, while many empirical studies reveal
that real data are seldom Gaussian. Despite this, the literature is
scarce in robust estimation techniques for network localization. Some
of the few approaches rely on identifying outliers from regular data
and discarding them. An example is Ihler \textit{et
  al.}~\cite{IhlerFisherMosesWillsky2005}, which formulates network
localization as an inference problem in a graphical model. To
approximate an outlier process the authors add a high-variance
Gaussian to the Gaussian mixtures and employ non-parametric belief
propagation to approximate the solution. The authors assume a
particular probability distribution for outlier measurements. In the
same vein, Ash \textit{et al.}~\cite{AshMoses2005} employs the EM
algorithm to jointly estimate outliers and sensor positions, and Yin
\ea ~\cite{YinZoubirFritscheGustafsson2013} tackled robust
{\color{blue} localization in a wireless environment, assuming noise
  can only enter the process through propagation and not measurement
  hardware. Also, their method is designed for single source
  localization using a set of base stations, whereas we deal with the
  robust network localization problem. Their algorithms, EM- and
  JMAP-ML, performed} estimation of the position, mixture parameters,
and {\color{blue} assumed an} outlier noise model for unknown
propagation conditions, again under predetermined probability
distributions. {\color{blue} Further, Yin et al.\ assume that several
  measurements are taken of the same distance. In our setup only one
  measurement is collected, for each underlying distance between
  elements.}.

Alternatively, methods may perform a soft rejection of outliers, still
allowing them to contribute to the solution.  O\u{g}uz-Ekim
\ea~\cite{OguzGomesXavierOliveira2011} derived a maximum-likelihood
estimator for Laplacian noise and relaxed it to a convex program by
linearizing and dropping a rank constraint; they also proposed a
centralized algorithm to solve the approximated problem. Such
centralized solutions fail to scale with the number of nodes and
number of collected measurements. Forero and
Giannakis~\cite{ForeroGiannakis2012} presented a robust
multidimensional scaling based on regularized least squares, where the
regularization term was replaced by a convex function, and solved via
MM. The main drawbacks of this approach are the centralized processing
architecture and selection of a sensitive regularization parameter.
Korkmaz and van der Veen~\cite{KorkmazVeen2009} use the Huber
loss~\cite{huber1964} composed with a discrepancy function between
measurements and estimated distances, in order to achieve robustness
to outliers. The resulting cost is nonconvex, and optimized by means
of the Majorization-Minimization technique. The method is distributed,
but the quality of the solution depends on the quality of the
initialization. {\color{blue} The recent work of Chen and
  colleagues~\cite{Chen2019} puts forward a new model for LOS/NLOS
  based on a multiplicative transformation of the additive data model,
  considering Exponential noise. The authors argue that this type of
  noise is routinely found in dense urban areas. They also propose a
  relaxation of their proposed data model to be run distributedly via
  the ADMM framework---unlike our present work, where we are focusing
  on the standard noise model-agnostic Huber M-estimator.} Yousefi
\ea~\cite{YousefiChangChampagne2014} { present a distributed gradient
  algorithm for a Huber discrepancy relaxed problem. This is inspired
  by the Projection Onto Convex Sets approach of Blatt and
  Hero~\cite{BlattHero2006} for single source localization, later
  extended to network localization by Gholami
  \ea~\cite{Gholami2011}. We independently developed the same
  relaxation (see also~\cite{SoaresGomes2014}), but whereas Yousefi
  \ea\ optimized it by means of a gradient method, our reformulation
  and analysis presented here leads to (i) a synchronous algorithm
  with optimal convergence, and (ii) a provably convergent
  asynchronous algorithm. These two modes are of paramount importance
  in real-world implementations, where fast methods clearly have an
  edge in the presence of time and power constraints, and the
  unstructured nature of engineered networks calls for asynchronous
  operation.}

\subsection{Contributions}
\label{sec:contributions}

In applications of large-scale networks there is a need for
distributed localization methods with soft rejection of outliers that
are simple to implement, scalable and efficient under reasonable
outlier noise distributions. Our methods incorporate outliers into the
estimation process and do not assume any statistical outlier model. We
capitalize on the robust estimation properties of the Huber function
but, unlike Korkmaz and van der Veen~\cite{KorkmazVeen2009}, we do not
address the nonconvex cost in our proposal, thus removing the
initialization uncertainty. Instead, we derive a convex relaxation
which numerically outperforms {relevant benchmark} methods, and other
natural formulations of the problem. The contributions of this work
are:
\begin{enumerate}
\item We motivate a \textbf{tight convex underestimator} for each term of the
  robust discrepancy measure for sensor network localization
  (Section~\ref{sec:conv-under}), and we provide an \textbf{optimality bound}
  for the convex relaxation. Further, we analyze the tightness of the
  convex approximation. We also compare it with other discrepancy
  measures and appropriate relaxations. All measurements contribute to
  the estimate, although we do not compute specific weights. Numerical
  simulations illustrate the quality of the convex underestimator
  (Section~\ref{sec:appr-qual-cvx-underestimator});
\item We develop a \textbf{distributed gradient method}, requiring only simple
  computations at each node, with \textbf{guaranteed optimal convergence
  rate} (Sections~\ref{sec:synchr-algor}
  and~\ref{sec:analysis-synchr-algor}). The keystone of the algorithm
  is a new representation of the Huber function composed with a norm
  (Section~\ref{sec:repr-huber-comp}). 
\item Further, we introduce an \textbf{asynchronous method} for robust network
  localization, with convergence guarantees
  (Sections~\ref{sec:asynchr-algor}
  and~\ref{sec:analysis-asynchr-algor});
\end{enumerate}
We benchmark our algorithms {against published methods} in robust
network localization, achieving better performance with fewer
communications (Section~\ref{sec:numer-exper}).  The proposed scheme
has optimal convergence rate for a first-order method, in the
synchronous case, without degradation of accuracy. It does not require
further approximation/relaxation for parallel operation, as it is
naturally distributed. Both properties follow from our novel
representation of the robust localization problem, as described in
Sections~\ref{sec:discrepancy-measure}
and~\ref{sec:conv-under}. {\color{blue} The reduced communication cost for the
distributed methods is due to having a cost whose gradient is
naturally distributed and the only data needed to broadcast to
neighbors is the current estimate of each node's position. At each
iteration a node only communicates a vector of the same size as the
ambient space ($\reals^2$ or $\reals^3$).}

\section{Discrepancy measure}
\label{sec:discrepancy-measure}

The network is represented as an undirected
graph~$\mathcal{G} = (\mathcal{V},\mathcal{E})$,
where~$\mathcal{V} = \{1,2, \dots, n\}$ indexes the set of sensors
with unknown positions and ${\mathcal E}$ is the set of edges. There
is an edge $i \sim j \in {\mathcal E}$ between nodes $i$ and $j$ if a
pairwise range measurement is available and~$i$ and $j$ can
communicate with each other. Anchors have known positions and are
collected in the set ${\mathcal A} = \{ 1, \ldots, m \}$; they are not
nodes on the graph~$\mathcal{G}$. For each sensor
$i \in {\mathcal V}$, we let ${\mathcal A}_i \subset {\mathcal A}$ be
the subset of anchors with measured range to node~$i$. The
set~$\mathcal{N}_{i}$ collects the neighbor sensor nodes of node~$i$.

The element positions belong to~$\reals^p$, with~$p=2$ for planar
networks and $p=3$ for volumetric ones.  We denote by
$x_i \in \reals^p$ the position of sensor $i$, and by $d_{ij}$ the
range measurement between sensors $i$ and $j$. The variable
$x=\{x_{\mathcal{V}}\}$ concatenates all unknown sensor
positions. Anchor positions are denoted by $a_{k} \in \reals^{p}$. We
let $r_{ik}$ denote the noisy range measurement between sensor $i$ and
anchor $k$. { The given data are, thus, the anchor
  positions~$\{a_{k} \in \reals^p \colon k \in \mathcal{A}\}$, noisy
  ranges between sensors~$\{d_{ij} \colon i \sim j \in \mathcal{E}\}$
  and noisy ranges between sensors and
  anchors~$\{r_{ik} \colon i \in \mathcal{V}, k \in
  \mathcal{A}_{i}\}$.}

We wish to estimate the sensor positions
taking into account two types of noise: (1) regular Gaussian noise,
and (2) outlier-induced noise.

Outlier measurements are non-Gaussian and will heavily bias the
solutions of the standard maximum-likelihood estimator for sensor
positions with additive independent and identically distributed
(i.i.d.) Gaussian noise, since their magnitude will be amplified by
the squares in each outlier term.  Robust estimation theory provides
some alternatives to perform soft rejection of outliers, namely,
replacing each quadratic loss above with the~$L_{1}$ loss or the Huber
loss, defined as
\begin{equation}
\label{eq:huber-loss}
h_{\delta}(u) =
\begin{cases}
  u^{2} & \text{if } |u| \leq \delta,\\
  2\delta|u|-\delta^{2} & \text{if } |u| \geq \delta.
\end{cases}
\end{equation}
The Huber loss achieves the best of two worlds: it is robust for large values
of the argument --- like the~$L_{1}$ loss --- and for reasonable noise
levels it behaves like a quadratic, thus leading to the maximum-likelihood
estimator adapted to regular noise{\color{blue}~\cite{zoubir2012robust}}.
\begin{figure}[tb]
  \centering
  \subfigure[For
         Gaussian noise,~$g_{Q}(x)$,
        evidences steep tails acting as outlier amplifiers. The
       loss~$g_{|\cdot|}(x)$, adapted to
    impulsive noise, fails to model the Gaussianity of
    low-power operating noise. The Huber
    loss~{$g_{R}(x)$} combines robustness to
  high-power outliers and adaptation to regular
  noise.]
    {\includegraphics[width=0.47\columnwidth]{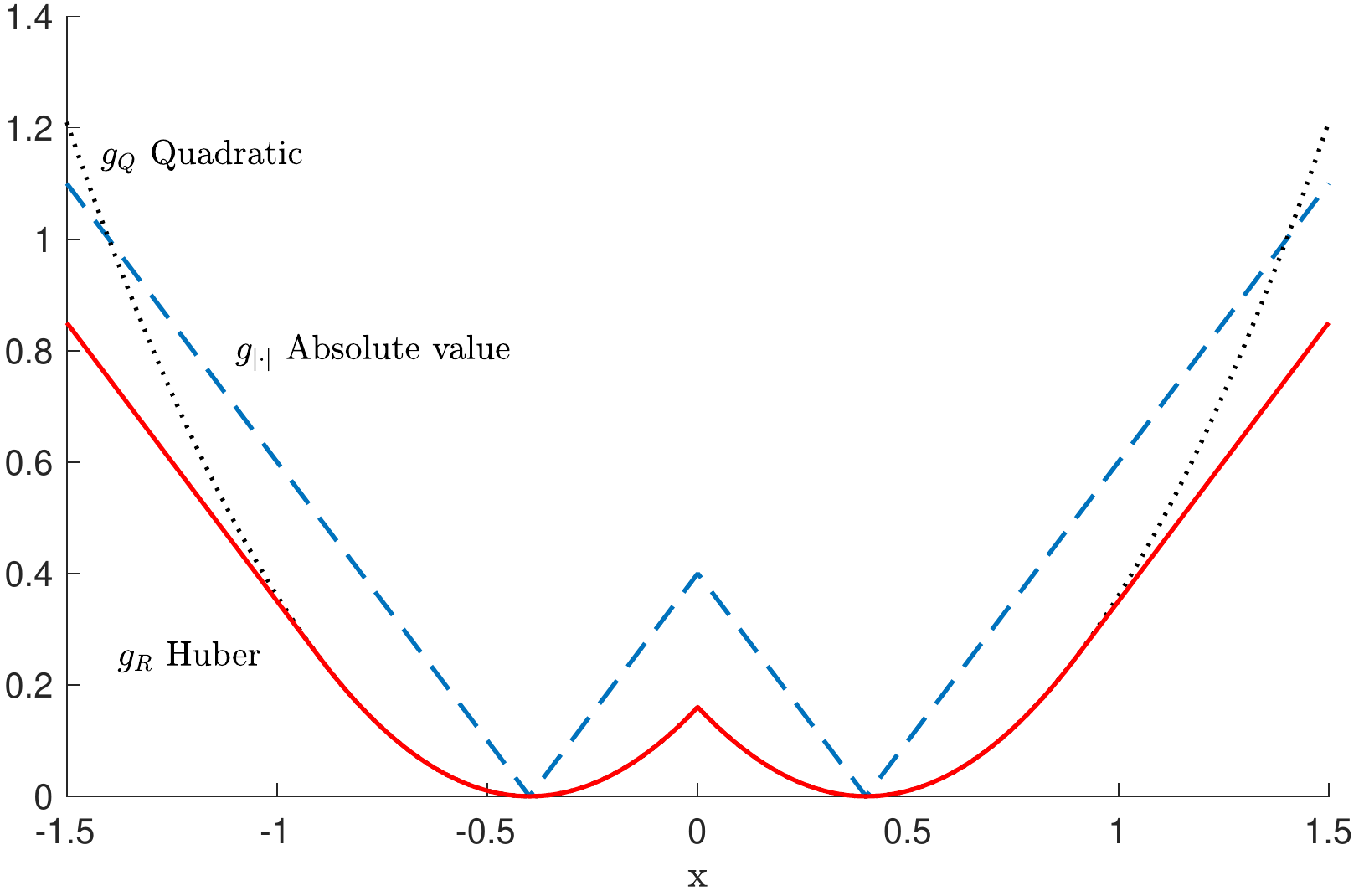}
      \label{fig:nonconvex}} \hfill
\subfigure[All functions~$f$ are the convex
    envelopes and, thus, the best convex approximations of the
    functions~$g$ in Figure~\ref{fig:nonconvex}. The convexification is
    performed by restricting the arguments of~$g$ to be nonnegative.]
  {\includegraphics[width=0.47\columnwidth]{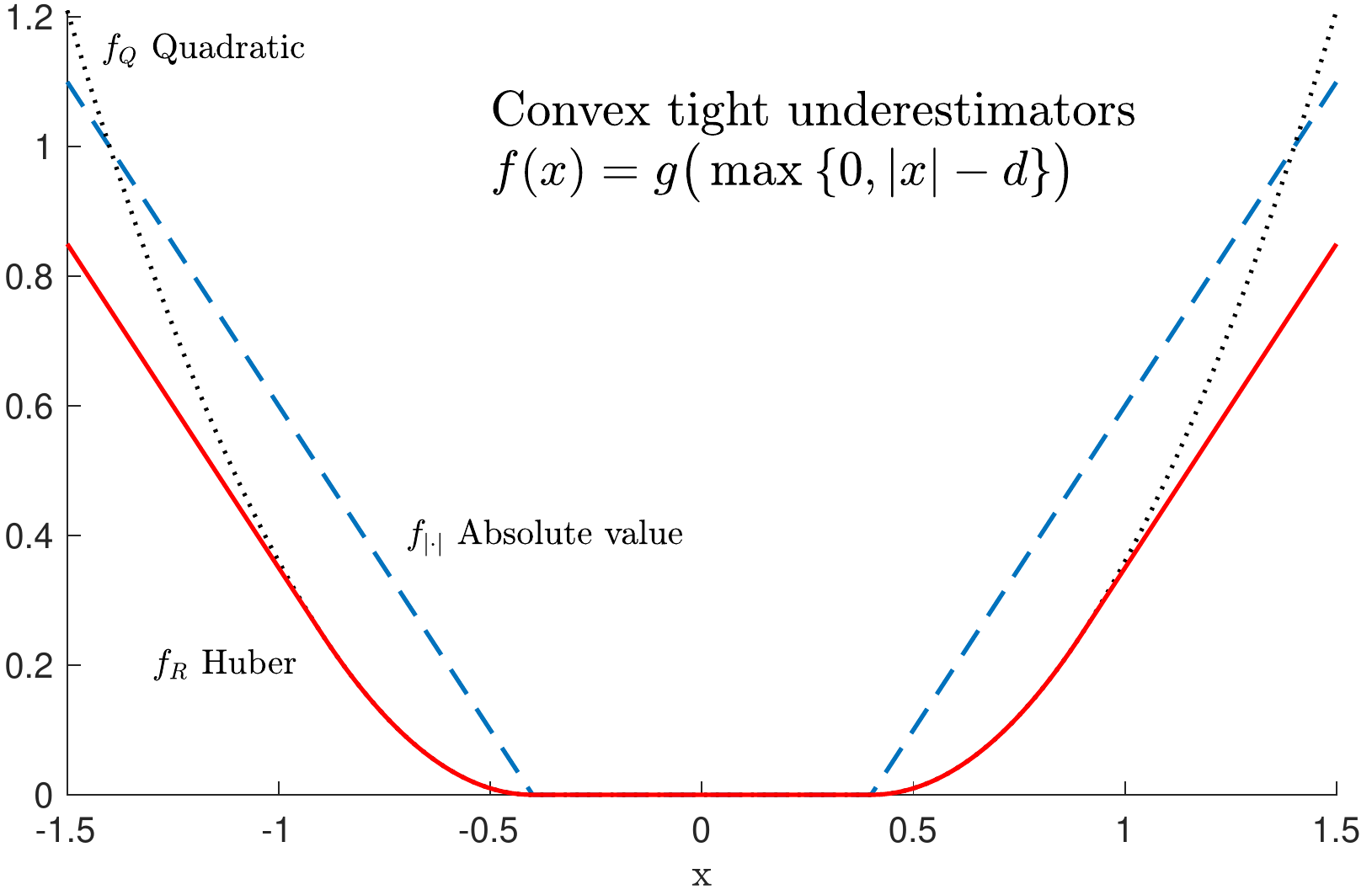}
  \label{fig:convexified}}
\caption{The different cost functions considered in this paper,
      applied to a 1D problem with  one node and one anchor.
        The anchor sits at the origin, and the node at 0.4,
        so~$d=0.4$ in the absence of noise. The scalar variable for the
        node position is denoted as~$x$. The ML estimator for independent white Gaussian noise is~$g_{Q}(x) = (|x| - d)^{2}$, while the ML estimator considering Laplacian noise is~$g_{|\cdot|}(x) = ||x| - d|$. The Huber M-estimator is defined by~{$g_{R}(x) = h_{\delta = 0.8}(|x| - d)$}.}
\end{figure}
Figure~\ref{fig:nonconvex} depicts the different {nonconvex
  costs~$g_{Q}$, $g_{|\cdot|}$, $g_{R}$ for Gaussian, $L_{1}$ and
  Huber losses, respectively, evaluated for a one-dimensional problem
  with one node and one anchor}. We can observe in this simple example
the main properties of the different cost functions, in terms of
adaptation to low/medium-power Gaussian noise and high-power outlier
spikes.  Using~\eqref{eq:huber-loss} we can write our robust
optimization problem as
\begin{equation}
  \label{eq:snlOptProb}
  \minimize_{x} g_{R}(x)
\end{equation}
where
\begin{align}
  \label{eq:huberDist}
    g_{R}(x) = & \sum _{i \sim j} \frac 12 h_{D_{ij}}(\|x_{i} - x_{j}\| - d_{ij})
     + 
    \sum_{i} \sum_{k \in \mathcal{A}_{i}} \frac 12 h_{R_{ik}}(\|x_{i}-a_{k}\| - r_{ik})
\end{align}
and $D_{ij}$, $R_{ik}$ denote the chosen Huber radii for inter-node
and node-anchor terms, which will be discussed in Section
\ref{sec:numer-exper}.  This function is nonconvex and, in general,
difficult to minimize.  We shall provide a convex underestimator that
tightly bounds each term of~\eqref{eq:huberDist}, thus leading to
better estimation results than other relaxations such
as~\cite{SimonettoLeus2014}. {\color{blue} The expected enhanced
  performance of our relaxation for the estimation under outlier noise
  comes naturally from the fact that our formulation directly accounts
  for it, whereas the one in~\cite{SimonettoLeus2014} does
  not. The mismatch between the data model and data distribution can strongly
  impact the quadratic terms considered in the ML estimator
  for Gaussian noise. A thorough study of the impact of robust
  statistics in Signal Processing can be found
  in~\cite{zoubir2012robust}.}

\section{Convex underestimator}
\label{sec:conv-under}
To convexify~$g_{R}$ we can replace each term by its
convex hull\footnote{The convex
  hull of a function~$\gamma$ is its best
  possible convex underestimator, defined as $ \text{conv } \gamma(x) =
  \sup\left \{ \eta(x) \; : \; \eta \leq \gamma, \; \eta \text{ is
      convex} \right \} $. It is hard to determine in general~\cite{UrrutyMarechal1993}.},
as depicted in Figure~\ref{fig:convexified} { following the same
  approach as in Falk and Soland~\cite{FalkSoland1969} or, more
  recently, in our paper~\cite{soares2014simple}}. Here, we observe
that the high-power behavior is maintained, whereas the
medium/low-power is only altered in the convexified area. We define
the convex costs by composing any of the convex losses with a
non-decreasing function~{$(\xi)_{+} = \max \{ 0, \xi\}$ which, in turn,
  operates on the discrepancies~$\|x_{i} - x_{j}\| - d_{ij}$,
  and~$\|x_{i} - a_{k}\| - r_{ik}$.
As~$\left ( \|x_{i} - x_{j}\| - d_{ij} \right )_{+}$
and~$\left ( \|x_{i} - a_{k}\| - r_{ik} \right )_{+}$} are
non-decreasing and each one of the { functions~$h_{D_{ij}}$
and~$h_{R_{ik}}$ is convex, then their sum
\begin{equation}\label{eq:huber-cvx}
\begin{aligned}
  f_{R}(x) = &\sum _{i \sim j} \frac 12 h_{D_{ij}}\left(\left(\|x_{i} - x_{j}\| -
      d_{ij}\right)_{+} \right) + 
  \sum_{i} \sum_{k \in \mathcal{A}_{i}} \frac 12 h_{R_{ik}}\left(\left(\|x_{i}-a_{k}\|
      - r_{ik}\right)_{+} \right)
\end{aligned}
\end{equation}
is also convex, for any positive value of parameters~$D_{ij}$ for
each edge~$i \sim j$, and~$R_{ik}$ for each node-anchor
pair~$i \in \mathcal{V}, k \in \mathcal{A}_{i}$.}  The cost
function~\eqref{eq:huber-cvx} also appears in Yousefi
\ea~\cite{YousefiChangChampagne2014} via a distinct reasoning. The
striking difference with respect to Yousefi \ea~is how the
cost~\eqref{eq:huber-cvx} is exploited here to generate distributed
solution methods where all nodes work in parallel {with
  optimal first-order convergence speed}, for the synchronous
algorithm, or randomly awaken, for the asynchronous algorithm,
{with proven convergence}.

\subsection{Approximation quality of the convex underestimator}
\label{sec:appr-qual-cvx-underestimator}

We now study the quality of the convexified estimator
using~(\ref{eq:huber-cvx}). The quadratic problem was addressed
in~\cite{soares2014simple}. We summarize the results here for the
reader's convenience and extend them to the $L_1$ and Huber convex
problems.

The optimal value of the cost {for any nonconvex~$g$ in
  $\{g_{Q}, g_{|\cdot|}, g_{R}\}$}, denoted by~$g^{\star}$, is bounded
by~$f^{\star} = f(x^{\star}) \leq g^{\star} \leq g(x^{\star})$,
where~$x^{\star}$ is the minimizer of the convex underestimator~$f$ {
  in~$\{f_{Q},f_{|\cdot|},f_{R}\}$, respectively},
and~$f^{\star} = \min_{x} f(x)$, is the minimum of function~$f$. A
bound for the optimality gap is,
thus,~$ g^{\star} - f^{\star} \leq g(x^{\star}) - f^{\star}$.  It is
evident that in all cases (quadratic, absolute value, and Huber)~$f$
is equal to~$g$ when~$\|x_{i} - x_{j}\| \geq d_{ij}$
and~$\|x_{i}-a_{k}\| \geq r_{ik}$. When the function terms differ,
say, for edges\footnote{The same reasoning applies to anchor terms.}
in a subset~$i \sim j \in \mathcal{E}_{2} \subset \mathcal{E}$, we
necessarily have~$\left(\|x_{i} - x_{j}\| - d_{ij}\right)_{+} = 0$,
leading to\footnote{Note that the minimizers in these expressions are
  not necessarily identical, but for simplicity we have used a common
  notation for the three losses.} the bounds
$ g_{Q}^{\star} - f_{Q}^{\star} \leq \sum_{i \sim j \in
  \mathcal{E}_{2}} \frac 12 \left(\|x_{i}^{\star} - x_{j}^{\star}\| -
  d_{ij} \right)^{2}$ for the quadratic loss,
$g_{| \cdot |}^{\star} - f_{| \cdot |}^{\star} \leq \sum_{i \sim j \in
  \mathcal{E}_{2}} \frac 12 \left| \|x_{i}^{\star} - x_{j}^{\star}\| - d_{ij}\right|$ for the absolute value loss, and
\begin{equation}
  \label{eq:bound-tighth}
  g_{R}^{\star} - f_{R}^{\star} \leq  \sum_{i \sim j \in
    \mathcal{E}_{2}} \frac 12 h_{D_{ij}}\left(\|x_{i}^{\star} -
    x_{j}^{\star}\| - d_{ij} \right)
\end{equation}
for the Huber loss.  
These bounds are an optimality gap guarantee available after the
convexified problem is solved; they tell us how low our estimates can
bring the original cost. Our bounds are tighter than the ones
available \textit{a priori} from applying~\cite[Th. 1]{udellBoyd2014},
which are~$g_{Q}^{\star} - f_{Q}^{\star} \leq  \sum_{i \sim j} \frac 12
  d_{ij}^{2}$ for the quadratic case,~$g_{| \cdot |}^{\star} - f_{| \cdot |}^{\star} \leq \sum_{i \sim j}
  \frac 12 d_{ij}$ for the absolute value problem, and
\begin{equation}
  \label{eq:udellBoydh}
  g_{R}^{\star} - f_{R}^{\star} \leq \sum_{i \sim j} \frac 12
  h_{D_{ij}}\left( d_{ij} \right),
\end{equation}
for the relaxation presented in this paper.
\begin{table}[t]
  \caption{Bounds on the optimality gap for a 1D example.} 
  \label{tab:bounds}
  \centering
  \begin{tabular}[h]{@{}p{1.5cm}c|ccc@{}}
    \toprule
    & 
    \textbf{True gap $g^{\star} - f^{\star}$} &
    \textbf{Proposed bound} &
    \textbf{\emph{A priori} bound \cite{udellBoyd2014}}\\\midrule
    Quadratic $g_{Q}^{\star} - f_{Q}^{\star}$ 
         &3.7019&5.5250&11.3405\\
    \rule{0pt}{4ex}   Absolute value $g_{| \cdot |}^{\star} - f_{| \cdot |}^{\star}$ 
         &1.1416&1.1533&3.0511\\
    \rule{0pt}{4ex}   Robust Huber $g_{R}^{\star} - f_{R}^{\star}$  &0.1784&0.1822&0.4786\\
    \bottomrule
  \end{tabular}
\end{table}
We show simulation results for a one-dimensional example averaged over
500 Monte Carlo trials in Table~\ref{tab:bounds}. The
true average gap $g^{\star} - f^{\star}$ is also listed. In the Monte
Carlo trials {we obtained a baseline noisy range measurement as
  $r_{k} = | \|x^{\star} - a_{k}\| + \nu_{k} |$, where~$x^{\star}$ is
  the true position and $\nu_{k}, k \in \mathcal{A}$ is a zero-mean
  (i.i.d.) Gaussian random variable with $\sigma = 0.04$}. One of the
measurements is then corrupted by a zero-mean random variable
with~$\sigma = 4$, modeling outlier noise. These results show the
tightness of the convexified function under such noisy conditions and
also demonstrate the looseness of the \textit{a priori} bound
in~\eqref{eq:udellBoydh}.
Not only do we use a robust dissimilarity, but we also add a smaller
optimality gap by using this particular surrogate. In the end, the
Huber-based approximation will be tighter, thus conferring robustness
to the estimator, as pointed out by Destino and
Abreu~\cite{DestinoAbreu2011}.

\section{Distributed and robust sensor network localization}
\label{sec:distr-robust-sens}

We construct our algorithm by reformulating~\eqref{eq:huber-cvx} as
the infimum of a sum of Huber functions composed with a norm, and then
by rewriting each of the terms with an alternative representation that
uncovers the possibility of a naturally distributed, optimal method
for estimating the unknown sensor positions. { More specifically, we
  will (A) reformulate the problem using a novel representation of the
  Huber function composed with a norm, (B) we will find the
  closed-form gradient for the new convex cost, and (C) we will verify
  that this gradient is Lipschitz continuous, minimized over a convex
  set. The latter is a key prerequisite for deriving a first-order
  method with optimal convergence rate.} For the first step, we invoke that
  each term of the first summation of~\eqref{eq:huber-cvx},
  corresponding to the edge~$i \sim j$, has a variational
  representation
  \begin{equation}
    \label{eq:cvx-term-norm}
    h_{D_{ij}}((\|x_{i}-x_{j}\| - d_{ij})_{+}) =\inf_{\|y_{ij}\|\leq d_{ij}} h_{D_{ij}}(\|x_{i}-x_{j}-y_{ij}\|),
  \end{equation}
  where~$y_{ij} \in \reals^{p}$ is an optimization variable.
The result~\eqref{eq:cvx-term-norm} follows from the variational
representation\footnote{This representation is related to the one used
  in \cite{OguzGomesXavierStosicOliveira2014}, to describe each one of
  the nonconvex terms in the ML estimator for Gaussian
  noise:~$(\|x_{i}-x{j}\| - d_{ij}) = \inf_{\|y\| = d_{ij}} \|x_i-x_j
  -y\|$.}~$(\|x_{i}-x{j}\| - d_{ij})_+ = \inf_{\|y\| \leq d_{ij}}
\|x_i-x_j -y\|$ already explored in previous
work~\cite{soares2014simple} and the fact that the Huber
function~$h_{D_{ij}}$ is monotonic on the positive
argument in~$\reals_+$.

\subsection{Alternative representation of the Huber function composed
  with a norm}
\label{sec:repr-huber-comp}

Using the variational representation from~\eqref{eq:cvx-term-norm}, we
can rewrite the convex unconstrained minimization
Problem~\eqref{eq:huber-cvx} as the constrained problem
\begin{equation}
  \label{eq:cvx-aux-variables}
  \begin{aligned}
    \minimize_{x,y,w} &\sum_{i \sim j}\frac 12 h_{D_{ij}}(\|x_{i}-x_{j}-y_{ij}\|) + \sum_{i} \sum_{k \in \mathcal{A}_{i}}
    \frac 12 h_{R_{ik}}(\|x_{i}-a_{k}-w_{ik}\|)\\
    \subjto &\left\{\|y_{ij}\| \leq d_{ij}, i \sim j\right\}, \; \left\{\|w_{ik}\| \leq r_{ik}, i \in \mathcal{V}, k \in
      \mathcal{A}_{i}\right\} ,
  \end{aligned}
\end{equation}
where we have introduced $y = \{y_{ij} : i \sim j\}$
and~$w = \{w_{ik} : i \in \mathcal{V}, k \in \mathcal{A}_{i}\}$.

We put forward a new representation of the Huber function
in~\eqref{eq:huber-loss}, when composed with the norm of a
vector,~{$\psi_{\delta}(u) = h_{\delta}(\|u\|)$}, {\color{blue} in the
  following proposition.}

\begin{proposition}[Representation of the Huber function composed with
  a norm]
  \label{prop:huber-composed}
  {\color{blue} Let $u \in \reals^q$, be a vector. Then, the composition of the
  Huber function with Huber radius~$\delta$ with the Euclidean norm
  of $u$, $\psi_{\delta}(u) = h_{\delta}(\|u\|)$ can be written as}
\begin{equation}
  \label{eq:huber-norm}
  \psi_{\delta}(u) = \|u\|^{2} - \mathrm{d}^{2}_{\delta}(u),
\end{equation}
{\color{blue} where we denote by~$\mathrm{d}^{2}_{\delta}(u)$ the
  squared Euclidean distance of
  vector~$u$ to a ball of radius~$\delta$ centered at the origin.}
\end{proposition}
\begin{proof}
{\color{blue} If the norm of the vector is smaller than the the Huber
  radius $\delta$, the distance of the vector to the ball centered in
  the origin with radius $\delta$ is zero, and thus we recover the
  first branch in~\eqref{eq:huber-loss},
  $ h_{\delta}(\|u\|) = \|u\|^2 \text{ if } \|u\| \leq
  \delta$. Otherwise, if we have that $ \|u\| \geq \delta$, then the
  squared distance to the disk has a closed form, that is the norm of
  $u$ minus the radius of the disk, $\delta$. So, the squared distance
  to the disk is
  $\mathrm{d}^{2}_{\delta}(u) = \|u\|^2 -2 \delta \|u\| +
  \delta^2$. As we are subtracting $\mathrm{d}^{2}_{\delta}(u)$ to
  $\|u\|^2$ we obtain the second branch of~\eqref{eq:huber-loss},
  and, thus,~\eqref{eq:huber-norm}.}
\end{proof}
We use this
representation to rewrite the cost in
problem~\eqref{eq:cvx-aux-variables} as
\begin{equation*}
  \begin{aligned}
    \sum_{i \sim j} \frac 12 \psi_{D_{ij}}(x_{i}-x_{j}-y_{ij}) + \sum_{i \in \mathcal{V}} \sum_{k \in \mathcal{A}_{i}}\frac 12
    \psi_{R_{ik}}(x_{i}-a_{k}-w_{ik}),
  \end{aligned}
\end{equation*}
or, in abbreviated form,
\begin{equation}
  \label{eq:huber-cart-prod}
    \frac 12 \psi_{\mathcal{D}}(Ax-y) + 
    \frac 12 \psi_{\mathcal{R}}(Mx - \alpha -w),
\end{equation}
{where~$\mathcal{D}$ is a
Cartesian product of balls
\begin{align}
\mathcal{D} & = \prod\limits_{i \sim j} \mathcal{D}_{ij}, &
\mathcal{D}_{ij} & = \{x : \|x\| \leq D_{ij}\},
\end{align}
and similarly for node-anchor terms
\begin{align}
\mathcal{R} & = \prod\limits_{i \in \mathcal{V}, k \in
  \mathcal{A}_{i}} \mathcal{R}_{ik}, &
\mathcal{R}_{ik} & = \{x : \|x\| \leq R_{ik}\}.
\end{align}
For compactness we have overloaded the notation for $\psi$ in
\eqref{eq:huber-cart-prod}; when indexed by a vector of concatenated
coordinates and a Cartesian product of balls, it is shorthand for the
sum of contributions \eqref{eq:huber-norm} for each coordinate and for
the radius of the ball in the corresponding element of the Cartesian
product. The squared distance function indexed by concatenated
coordinates and a Cartesian product,
$\mathrm{d}^{2}_{\mathcal{D}}(\cdot)$ or
$\mathrm{d}^{2}_{\mathcal{R}}(\cdot)$, is also used below and
similarly defined in terms of the regular squared distance in
\eqref{eq:huber-norm}.}
Matrix $A = C \otimes I$ in~(\ref{eq:huber-cart-prod}) is the
Kronecker product of the arc-node incidence matrix\footnote{The
  arc-node incidence matrix $C$ is a
  $|\mathcal{E}| \times |\mathcal{V}|$ matrix. The rows and the
  columns of $C$ are indexed by $\mathcal{E}$ and $\mathcal{V}$,
  respectively. The $(e,i)$-entry of $C$ is 0 if node $i$ and edge $e$
  are not incident, and otherwise it is 1 or -1 according to the
  direction agreed on the operation onset by the two nodes involved in
  edge~$e$.}~$C$ associated with the graph~$\mathcal{G}$, and the
identity matrix with dimension of the ambient space (usually, 2 or 3).
The product $Ax$ thus concatenates coordinate differences
$x_{i}-x_{j}$ as required by the ordering of elements chosen for
vector $y$.  On the second term of \eqref{eq:huber-cart-prod} $M$ is a
selection matrix that concatenates coordinates $x_{i}$ according to
the ordering of elements chosen for vector $w$, whereas vector
$\alpha$ concatenates anchor coordinates $a_{k}$ with the same
criterion.  Consider the aggregation of variables~$z = (x,y,w)$; we
define the constraint set in~\eqref{eq:cvx-aux-variables} as
\begin{equation}
  \label{eq:constraint-set}
  \mathcal{Z} = \{(x,y,w) : \|y_{ij}\| \leq d_{ij}, i \sim j,
  \|w_{ik}\| \leq r_{ik}, k \in \mathcal{A}_{i}, i \in \mathcal{V}\},
\end{equation}
and the cost as
\begin{equation}
  \label{eq:matrix-cost}
  {F(z) := \frac 12 \psi_{\mathcal{D}}(Bz) +
  \frac 12 \psi_{\mathcal{R}}(Ez-\alpha),}
\end{equation}
where~$B = [A\quad -I \quad 0]$ and~$E = [M \quad 0 \quad -I]$.
With this notation,~\eqref{eq:cvx-aux-variables} becomes
\begin{equation}
  \label{eq:cvx-matrix-z}
  \begin{aligned}
    \minimize & \quad F(z)\\
    \subjto & \quad z \in \mathcal{Z}.
  \end{aligned}
\end{equation}
{ Problem~\eqref{eq:cvx-matrix-z} is convex
  and, as we will see next, can be solved by a first-order method with optimal convergence
  rate, namely, FISTA, proposed by Beck and
  Teboulle~\cite{BeckTeboulle2009}. The required properties
  for FISTA are: (1) Lipschitz continuous gradient of the cost~$F(z)$,
  (2) known Lipschitz constant, and (3) an easy way to compute the
  proximal operator of the balls' indicator function, as described in
  the following sections.}

\subsection{Gradient}
\label{sec:gradient}
To compute the gradient of the cost in~\eqref{eq:matrix-cost}, we need
the gradient of the squared distance to a convex set --- a known result from
convex analysis (see~\cite[Prop. X.3.2.2,
Th. X.3.2.3]{UrrutyMarechal1993}).
Let us denote the squared distance to a convex set~$C$ as
$  \phi(u) = \frac 12 \mathrm{d}^{2}_{C}(u)$.
 Then, from convex analysis, we know
that~$\phi$ is convex, differentiable, and its gradient is
$   \nabla \phi (u) = u - \mathrm{P}_{C}(u)$,
where~$\mathrm{P}_{C}(u)$ is the orthogonal projection of point~$u$
onto the set~$C$,~$\mathrm{P}_{C}(u) = \argmin_{y \in C} \|u-y\|$.
Knowing this, we can compute the gradient of~\eqref{eq:huber-norm}
as~$\nabla \psi_{R}(u) = 2u-2(u-\mathrm{P}_{R}(u)) = 2
\mathrm{P}_{R}(u)$, and the gradient of~\eqref{eq:matrix-cost} as
\begin{equation}
  \nabla F (z) = \frac 12 B^{\top} \nabla  \psi_{\mathcal{D}}(Bz) +
  \frac 12
  E^{\top} \nabla  \psi_{\mathcal{R}}(Ez-\alpha)\label{eq:gradient}
  = B^{\top} \mathrm{P}_{\mathcal{D}}(Bz) + E^{\top}
  \mathrm{P}_{\mathcal{R}}(Ez-\alpha).
\end{equation}
Above, to project a vector of concatenated coordinates onto a
Cartesian product of balls amounts to concatenating the projection of
each coordinate onto the corresponding ball. This is used explicitly
in the given summary of the synchronous
Algorithm~\ref{alg:synchronous}.

\subsection{Lipschitz constant}
\label{sec:lipschitz-constant} 

{ Finding a differentiable function such as~$F(z)$, whose
  gradient admits a Lipschitz constant~$L_{F}$ such that
  \begin{equation*}
    \| \nabla F(z_{1}) -\nabla F(z_{2}) \| \leq L_{F} \|z_{1} -
    z_{2}\|, \text{ for all } z_{1}, z_{2},
  \end{equation*}
  means that we have access to~$L_F$ an upper bound to the curvature
  of the convex function~$F(z)$. The optimal gradient method
  FISTA~\cite{BeckTeboulle2009} can be used once this upper bound is
  known. We will now obtain such an upper bound~$L_{F}$ with the
  favourable property that the network can compute it in a distributed
  way where the complexity per node does not increase as the network
  grows.}

It is widely known that
projections onto convex sets shrink distances~\cite{Phelps1957}, \ie,
\begin{equation}
  \label{eq:project-shrink-dist}
  \|\mathrm{P}_{C}(u) - \mathrm{P}_{C}(v)\| \leq \|u - v\|.
\end{equation}
Using this result, we can compute a Lipschitz constant
for~\eqref{eq:gradient}. First, let us focus on the inter-node term
in~\eqref{eq:gradient}:
\begin{equation*}
  \begin{split}
    \|B^{\top}\mathrm{P}_{\mathcal{D}}(Bu) -
    B^{\top}\mathrm{P}_{\mathcal{D}}(Bv)\| &\leq
    \mtxnorm{B}\|\mathrm{P}_{\mathcal{D}}(Bu) -
    \mathrm{P}_{\mathcal{D}}(Bv)\|\\ & \leq \mtxnorm{B}\|Bu-Bv\| \leq
    \mtxnorm{B}^{2}\|u-v\| \\ &= \lambda_{\mathrm{max}}(BB^{\top})\|u-v\|,
  \end{split}
\end{equation*}
{where~$\mtxnorm{B}=\sigma_{\mathrm{max}}(B)=\sqrt{\lambda_{\max}(BB^\top)}$ is the matrix 2-norm of~$B$.} The
first inequality follows from the Cauchy-Schwarz inequality, the
second from \eqref{eq:project-shrink-dist}, and the final one from
Cauchy-Schwarz again. The maximum eigenvalue of~$BB^{\top}$ can be
bounded by
\begin{equation}\label{eq:inter-node-lips}
  \begin{split}
    \lambda_{\mathrm{max}}(BB^{\top}) & = \lambda_{\mathrm{max}}\left(
      \begin{bmatrix}
        A& -I& 0
      \end{bmatrix}
      \begin{bmatrix}
        A^{\top}\\ -I \\ 0
      \end{bmatrix}
    \right)\\ & = \lambda_{\mathrm{max}}(AA^{\top} +I) \\ &= (1+
    \lambda_{\mathrm{max}}(AA^{\top}))\\ & = (1+
    \lambda_{\mathrm{max}}(L)) \\ & \leq (1+ 2\delta_{\mathrm{max}}),
  \end{split}
\end{equation}
where~$L$ is the graph Laplacian matrix, and~$\delta_{\max}$ is the maximum
node degree of the network\footnote{To characterize the network we use
  the concepts of \emph{node degree}~$k_{i}$, the number of edges that
  touch a node~$v$. Maximum node degree is the maximum of this number
  for all nodes in the network, , and \emph{average node
    degree}~$\langle k \rangle = 1/n \sum_{i=1}^{n}k_{i}$.}. A proof
of the inequality~$\lambda_{\mathrm{max}}\leq 2\delta_{\mathrm{max}}$
is given in Bapat~\cite{Bapat2010}. In the same way, for the
node-anchor terms, we have
\begin{equation*}
{  \|E^{\top}\mathrm{P}_{\mathcal{R}}(Eu) -
  E^{\top}\mathrm{P}_{\mathcal{R}}(Ev)\|}  \leq 
  \mtxnorm{E}^{2}\|u-v\| 
   =  \lambda_{\mathrm{max}}(EE^{\top})\|u-v\| 
\end{equation*}
and the constant can be upper-bounded by
\begin{equation} \label{eq:node-anchor-lips}
  \begin{split}
    \lambda_{\mathrm{max}}(EE^{\top}) & = \lambda_{\mathrm{max}}\left(
      \begin{bmatrix}
        M 0 -I
      \end{bmatrix}
      \begin{bmatrix}
        M^{\top}\\ 0\\-I
      \end{bmatrix}
    \right) \\ & = \lambda_{\mathrm{max}}\left(MM^{\top}+I\right)\\ &\leq
    \left(1+\lambda_{\mathrm{max}}(MM^{\top})\right)\\ & \leq \left(1 +
      \max_{i \in \mathcal{V}} |\mathcal{A}_{i}|\right),
  \end{split}
\end{equation}
where the last inequality can be deduced from the structure of
matrix~$M$: there are identity matrices for each node~$i$ for as many
anchor measurements it collected. From~\eqref{eq:inter-node-lips}
and~\eqref{eq:node-anchor-lips} we can see that a Lipschitz constant
for~\eqref{eq:matrix-cost} is
\begin{equation}
  \label{eq:lipschitz}
  L_{F} = 2 + 2\delta_{\mathrm{max}} + \max_{i \in \mathcal{V}} |\mathcal{A}_{i}|.
\end{equation}
We stress that this constant is small, does not depend on the size of
the network, {and can be computed in a distributed way, for example by
  flooding}~\cite{chung1997spectral}.

\subsection{Synchronous algorithm}
\label{sec:synchr-algor}

The gradient in~\eqref{eq:gradient} and its Lipschitz continuity, with
the constant in~\eqref{eq:lipschitz}, enable us to use Nesterov's
optimal gradient method~\cite{Nesterov1983, Nesterov2004}, further
developed by Beck and Teboulle \cite{BeckTeboulle2009}.  Firstly, we
must write the problem as an unconstrained minimization using an
indicator function~$I_{\mathcal{Z}}(u)=
\begin{cases}
  0 \qquad \text{if } u \in \mathcal{Z}\\ +\infty \quad \text{otherwise }
\end{cases}
$, and incorporate the constraints in the problem formulation { to obtain the problem}
$  {\minimize_{z} F(z) + I_{\mathcal{Z}}(z).}$
Then we
perform the proximal minimization of the unconstrained problem
{
using FISTA. For completeness, we outline FISTA, whose optimality convergence
proofs can be found in the aforementioned work of Beck and
Teboulle~\cite{BeckTeboulle2009}. For any initialization~$z[0] =
z[-1]$ and for~$t\geq 1$ the FISTA algorithm performs the following
iterations:}
\begin{equation*}
  \begin{split}
    {\zeta} & {= z[t-1] + \frac{t-2}{t+1} \big( z[t-1] - z[t-2] \big )},
    \\ {z[t] } & {= \text{Prox}_{I_{\mathcal{Z}}} \big ( \zeta -
      \frac{1}{L_{F}}\nabla F(\zeta) \big )},
  \end{split}
\end{equation*}
{where~$\text{Prox}_{I_{\mathcal{Z}}}(\cdot)$ is the
proximal operator of function~$I_{\mathcal{Z}}$,
that, for the present case, can be computed as}
\begin{equation*}
\begin{split}
  {\text{Prox}_{I_{\mathcal{Z}}}(z)}  &{ = \argmin_{u} \left (
      I_{\mathcal{Z}}(u) + \frac 12 \|u-z\|^{2} \right )}\\
     &{ = \mathrm{P}_{\mathcal{Z}}(z),}
\end{split}
\end{equation*}
{where~$\mathrm{P}_{\mathcal{Z}}(z)$ is the projection of
  point~$z$ onto set~$\mathcal{Z}$.}
The result for our reformulation is shown in Algorithm~\ref{alg:synchronous}, {where iteration counts are given in square
brackets.}
\begin{algorithm}[tb]
  \caption{Synchronous method: STRONG}
  \label{alg:synchronous}
  \begin{algorithmic}[1] 
    \REQUIRE
    $L_{F}; \{d_{ij}, \: D_{ij} : i \sim j \in \mathcal{E}\}; \{r_{ik}, \: R_{i,k} : i \in
    \mathcal{V}, k \in \mathcal{A}_{i}\}$;
    \ENSURE $\hat x$
    \STATE let $ \mathcal{Y}_{ij} = \{y \in \reals^{p}: \|y\| \leq d_{ij}\} $; and $ \mathcal{W}_{ik} = \{w \in \reals^{p}: \|w\| \leq r_{ik}\} $;
    \STATE each node~$i$ chooses and broadcasts
    arbitrary~$x_{i}{[0]} = x_{i}{[-1]}$; 
    \STATE set
    $y_{ij}{ [0]} = \mathrm{P}_{\mathcal{Y}_{ij}}\left(x_{i}{ [0]}
      -x_{j}{ [0]}\right)$;
    and
    $w_{ik}{ [0]} = \mathrm{P}_{\mathcal{W}_{ik}}\left(x_{i}{ [0]}
      -a_{k}\right)$;
    \STATE $t = 0$; 
    \WHILE{some stopping criterion is not met, each
      node~$i$} 
    \STATE $t = t+1$; 
    \STATE
    $\xi_{i} = x_{i}{ [t-1]} + \frac{t-2}{t+1} \left(x_{i}{ [t-1]} -
      x_{i}{ [t-2]} \right)$;\label{alg:xi}
    \STATE broadcast~$\xi_{i}$ to all neighbors { and
      listen for neighbors' $\xi_{j}$}; 
    \FOR{all~$j$ in the neighbor set~$\mathcal{N}_{i}$} 
    \STATE
    $\upsilon_{ij} = y_{ij}{ [t-1]} + \frac{t-2}{t+1} \left(y_{ij}{ [t-1]}
      - y_{ij}{ [t-2]} \right)$;\label{alg:upsilon}
    \STATE
    $ y_{ij}{ [t]} = \mathrm{P}_{\mathcal{Y}_{ij}} \left( \upsilon_{ij}
      + \frac{1}{L_{F}} \mathrm{P}_{\mathcal{D}_{ij}}\left(\xi_{i}-\xi_{j} -
        \upsilon_{ij}\right)\right) $;
    \ENDFOR
    \FOR{all~$k$ in the anchor set~$\mathcal{A}_{i}$}
    \STATE
    $
    \omega_{ik} = w_{ik}{ [t-1]} + \frac{t-2}{t+1} \left(w_{ik}{ [t-1]} -
      w_{ik}{ [t-2]} \right)$;\label{alg:omega}
    \STATE  $w_{ik}{ [t]} = \mathrm{P}_{\mathcal{W}_{ik}}\left(
        \omega_{ik}+\frac{1}{L_{F}}\mathrm{P}_{\mathcal{R}_{ik}}(\xi_{i}-a_{ik}-\omega_{ik})\right)
    $;
    \ENDFOR
    \STATE  $
        {\nabla F_{i}} = \sum_{j \in \mathcal{N}_{i}}
        \mathrm{P}_{\mathcal{D}_{ij}}(\xi_{i}-\xi_{j}-\upsilon_{ij}) + \sum_{k \in \mathcal{A}_{i}} \mathrm{P}_{\mathcal{R}_{ik}}\left(\xi_{i}-a_{k}-\omega_{ik}\right);$
    \STATE $
          x_{i}{ [t]} =\xi_{i} - \frac{1}{L_{F}}{ \nabla F_{i}}
     $;\label{alg:x}
    \ENDWHILE
    \RETURN $\hat x_{i} = x_{i}{ [t]}$
  \end{algorithmic}
\end{algorithm}
We denote the entries of~$\nabla F$
regarding variable~$x_{i}$ as~${\nabla F_{i}}$.
Each node~$i$ will update:
\begin{itemize}
\item{the current estimate of its own
position,}
\item{each one of the~$y_{ij}$ for all the incident
edges, and}
\item{the anchor terms~$w_{ik}$, if any.}
\end{itemize}
In step~\ref{alg:xi} we have the extrapolation step for each~$x_{i}$,
whereas in steps~\ref{alg:upsilon} and~\ref{alg:omega} we can see the
update of the extrapolation steps for each one of the edge
variables~$y_{ij}$, and~$w_{ik}$, respectively. {\color{blue} In terms
  of computational complexity at each node, the operations are
  multiplications and summations, and thus do not demand intensive
  calculus and the corresponding battery drain. The same applies to
  the nonconvex Huber method~\cite{KorkmazVeen2009}. In the
  other extreme, the SDP method~\cite{SimonettoLeus2014} is addressed
  with interior point algorithms, that are cubic in the problem
  dimension.}

\paragraph*{Paralellizability}
We observe that each block
of~$z = (x, y, w)$ at iteration~$t$ only needs local neighborhood
information, as shown in Algorithm~\ref{alg:synchronous}. To
demonstrate the natural distribution of the method we go back
to~\eqref{eq:gradient}. Here, the
term~$E^{\top} \mathrm{P}_{\mathcal{R}}(Ez-a)$ only involves anchor
measurements relative to each node, and so it is distributed. The
term~$B^{\top}\mathrm{P}_{\mathcal{D}}(Bz)$ is less clear. The
vector~$Bz$ collects~$x_{i}-x_{j}-y_{ij}$ for all edges~$i \sim j$ and
to it we apply the projection operator onto the Cartesian product of
balls. This is the same as applying a projection of each edge onto
each ball. When left multiplying with~$B^{\top}$ we
get~$B^{\top}\mathrm{P}_{\mathcal{D}}(Bz)$. The left multiplication
by~$B^{\top}$ will group at the position of each node variable~$x_{i}$
the contributions of all incident edges to node~$i$.
\paragraph*{Edge variables}
Recall that we introduced edge variables $y_{ij}$ in
\eqref{eq:cvx-term-norm} to reformulate terms in the
maximum-likelihood function in a more tractable form. These are
associated with edges linking pairs of nodes in the measurement graph
$\mathcal{G}$. Variables $w_{ik}$ play the same role for node-anchor
terms.  To update the~$y_{ij}$ variables we could designate one of the
incident nodes, $i$ or $j$, as responsible for the update and then
communicate the result to the non-computing neighbor. But, to avoid
this expensive extra communication, we decide that each node~$i$
should compute its own~$y_{ij}$, where~$y_{ij} = -y_{ji}$ for all
edges. Also, with this device, the gradient entry regarding
variable~$x_{i}$ would
be~$\sum_{j \in N_{i}} C_{(i\sim j,i)}
\mathrm{P}_{\mathcal{D}_{ij}}\left( C_{(i\sim j,i)}(x_{i}-x_{j} -
  y_{ij})\right)$.  The symbol $C_{(i \sim j,i)}$ denotes the arc-node
incidence matrix entry relative to edge~$i \sim j$ (row index) and
node~$i$ (column index).  As the projection onto a ball $\mathcal{B}$
of radius~$\delta$ centered at the origin can be written
as
\begin{equation*}
\mathrm{P}_{\mathcal{B}}(u) =
\begin{cases}
  \frac {u}{\|u\|}\delta & \qquad \text{if } \|u\| > \delta\\
  u & \qquad \text{if } \|u\| \leq \delta
\end{cases},
\end{equation*}
then~$\mathrm{P}_{\mathcal{B}}(-u) = - \mathrm{P}_{\mathcal{B}}(u)$, and, thus, the
gradient entry regarding variable~$x_{i}$
becomes~$\sum_{j \in \mathcal{N}_{i}} \mathrm{P}_{\mathcal{D}_{ij}}\left( x_{i}-x_{j} -
  y_{ij} \right)$, as stated in Algorithm~\ref{alg:synchronous}.

\paragraph{{\color{blue} Communications}} {\color{blue} For each
  iteration $t$, the algorithm requires only that each agent $i$
  broadcast to its neighbors one vector in the ambient space
  ($\reals^2$ or $\reals^3$) with the Nesterov extrapolated point
  $\xi_i$. The position estimates $x_i[t]$, node-anchor variables
  $w_{ik}[t]$, and edge variables $y_{ij}[t]$ are computed internally
  at each node.}

\subsection{Asynchronous algorithm}
\label{sec:asynchr-algor}
In Section~\ref{sec:synchr-algor} we presented a distributed method
addressing the robust network localization problem in a scalable
manner, where each node uses information from its neighborhood and
performs a set of simple arithmetic computations. But the results
still depend critically on synchronous computation, where nodes
progress in lockstep through iterations. As the number of processing
nodes becomes very large, this synchronization can become seriously
difficult --- and unproductive. An asynchronous approach is called for
in such very large-scale and faulty settings. In an asynchronous time
model, the nodes move forward independently and algorithms should withstand
certain types of faults, like temporary unavailability of a node.  To
address this issue, we present a fully asynchronous method, based on a
broadcast gossip scheme (c.f. Shah~\cite{shah2009} for an extended
survey of gossip algorithms).

Nodes are equipped with independent clocks ticking at random times
(say, as Poisson point processes).
When node~$i$'s clock ticks, it performs the update of its
variables and broadcasts the update to its neighbors. Let the
order of node activation be collected
in~$\{\chi_{t}\}_{t \in \naturals}$, a sequence of independent random
variables taking values on the set~$\mathcal{V}$, such that {the
probability~$\prob(\chi_{t} = i)$ is always positive, i.e.,}
\begin{equation}
  \label{eq:rv}
  \prob(\chi_{t} = i) = P_{i} > 0.
\end{equation}
This assumption ensures that the node can be activated.  The
asynchronous update of variables on node~$i$ is summarized in
Algorithm~\ref{alg:asyncronous} and described in detail below.
\begin{algorithm}[tb]
  \caption{Asynchronous method: asyncSTRONG}
  \label{alg:asyncronous}
  \begin{algorithmic}[1]
    \REQUIRE $L_{F}; \{d_{ij}, \: D_{ij} : i \sim j \in \mathcal{E}\}; \{r_{ik}, \: R_{i,k} : i \in
    \mathcal{V}, k \in \mathcal{A}_{i}\}$;
    \ENSURE $\hat x$
    \STATE let $ \mathcal{Y}_{ij} = \{y \in \reals^{p}: \|y\| \leq d_{ij}\} $; and $ \mathcal{W}_{ik} = \{w \in \reals^{p}: \|w\| \leq r_{ik}\} $;
    \STATE each node $i$ chooses {and broadcasts} random $x_{i}{[0]}$;
    \STATE  \textbf{Initialization:} {set
    $y_{ij}{ [0]} = \mathrm{P}_{\mathcal{Y}_{ij}}\left(x_{i}{ [0]}
      -x_{j}{ [0]}\right)$;
    and
    $w_{ik}{ [0]} = \mathrm{P}_{\mathcal{W}_{ik}}\left(x_{i}{ [0]}
      -a_{k}\right)$;}
    \STATE $t = 0$;
    \WHILE{some stopping criterion is not met, each node $i$}
    \STATE $t = t + 1;$
    \STATE $
    \begin{aligned}[t]
      x_{i}{[t]} =
      \begin{dcases}
        \argmin_{\substack{\xi_{i}, \{y_{ij} \in
            \mathcal{Y}_{ij}, j \in \mathcal{N}_{i} \}, \\ \{w_{ik} \in \mathcal{W}_{ik}, k
          \in \mathcal{A}_{i}\}}} F_{i}(\xi_{i},\{y_{ij}\},\{w_{ik}\}) \label{alg:asynminx}
        &\mbox{if } \chi_{t} = i\\
        x_{i}{[t-1]} &\mbox{otherwise;}
      \end{dcases}
    \end{aligned}
    $
    \STATE if~{$\chi_{t}= i$}, broadcast~$x_{i}{[t]}$ to neighbors
    \ENDWHILE
    \RETURN $\hat x = x{[t]}$
  \end{algorithmic}
\end{algorithm}
It is useful to recast Problem~\eqref{eq:matrix-cost}--\eqref {eq:cvx-matrix-z} as
\begin{equation}
  \label{eq:per-node}
  \begin{aligned}
    \minimize_{x,y,w} &\sum_{i} \left(\sum_{j \in \mathcal{N}_{i}} \frac 14
      \|x_{i}-x_{j}-y_{ij}\|^{2} - \frac 14
      \mathrm{d}^{2}_{\mathcal{D}_{ij}}(x_{i}-x_{j}-y_{ij}) + \right.\\ 
      & \left.\sum_{k \in \mathcal{A}_{i}}
      \frac 12 \|x_{i}-a_{k}-w_{ik}\|^{2} - \frac 12
      \mathrm{d}^{2}_{\mathcal{R}_{ik}}(x_{i}-a_{k}-w_{ij}) \right)\\ 
      \subjto & \qquad \|y_{ij}\| \leq d_{ij}, \quad \|w_{ik}\| \leq r_{ik},,
  \end{aligned}  
\end{equation}
where the factor~$\frac{1}{4}$ accounts for the duplicate terms when
considering summations over nodes instead of over edges. Our strategy for obtaining an asynchronous algorithm simply consists in solving a single-source localization problem at any given node upon wake-up, fixing the
neighbor positions to their last known values. Node $i$ thus solves
\begin{equation}
  \label{eq:sl-problem}
  \begin{aligned}
    \minimize_{\substack{x_{i}, \{y_{ij}: \: j \in \mathcal{N}_{i} \}, \\ \{w_{ik}: \: k \in \mathcal{A}_{i}\}}} & \;{ F_{i}(x_{i},\{y_{ij}\},\{w_{ik}\})}\\
    \subjto & \qquad \|y_{ij}\| \leq d_{ij}, \quad \|w_{ik}\| \leq r_{ik},
  \end{aligned}
\end{equation}
where
\begin{equation*}
 F_{i}(x_{i},\{y_{ij}\},\{w_{ik}\}) := {\sum_{j \in \mathcal{N}_{i}} \frac 14
    \psi_{D_{ij}}(x_{i}-x_{j}-y_{ij}) + }{\sum_{k \in \mathcal{A}_{i}} \frac 12 \psi_{R_{ik}}(x_{i}-a_{k}-w_{ij}).}
\end{equation*}
Problem~\eqref{eq:sl-problem} is convex, solvable at each
node by a general purpose solver. Nevertheless, that approach would not
take advantage of the specific problem structure, thus depriving the
solution of an efficient and simpler computational procedure. Again,
Problem~\eqref{eq:sl-problem} can be solved by the Nesterov optimal
first-order method, because the gradient of~$F_{i}$ is Lipschitz
continuous {in~$x_{i}, y_{ij}$, and~$w_{ik}$}, accepting the
same Lipschitz constant as~$F$, in~\eqref{eq:lipschitz}.
\paragraph{{\color{blue} Communications}} {\color{blue} For each
  iteration $t$, the algorithm requires only that the agent
  $i = \chi_{t}$ broadcast to its neighbors a vector in the ambient
  space ($\reals^2$ or $\reals^3$) with the position estimate
  $x_i[t]$.}

\section{Convergence analysis}
\label{sec:analysis}
In this section we address the convergence of 
Algorithms~\ref{alg:synchronous} and~\ref{alg:asyncronous}. We
provide convergence guarantees and rate of convergence for the
synchronous version, and we also prove convergence for the
asynchronous method.
\subsection{Synchronous algorithm}
\label{sec:analysis-synchr-algor}
As shown in Section~\ref{sec:distr-robust-sens}, Problem~\eqref
{eq:cvx-matrix-z} is convex and the cost function has a Lipschitz
continuous gradient. As proven by
Nesterov~\cite{Nesterov1983, Nesterov2004}, and further
developed by Beck and Teboulle~\cite{BeckTeboulle2009},
Algorithm~\ref{alg:synchronous} converges at the optimal rate
$O\left( t^{-2} \right)$. Specifically,
$F( z^{t} ) - F^\star \leq \frac{2 L_{F}}{(t+1)^2} \left\| z^{0} -
  z^\star \right\|^2$,
where~$F^{\star}$ is the optimal value and~$z^{\star}$ is a minimizer
of Problem~\eqref {eq:cvx-matrix-z}.

\subsection{Asynchronous algorithm}
\label{sec:analysis-asynchr-algor}

To investigate the convergence of Algorithm~\ref{alg:asyncronous}, as
it is a constrained problem unlike~\cite{soares2014simple}, we will
use a totally different proof strategy, not by crafting a
supermartingale that bounds the gradient norm, but depending on the
Borel-Cantelli Lemma, as formally described in~\ref{sec:proof-theor-refth}. We need the following trivial
assumption for the network localization problem:
\begin{assumption}
  \label{th:connected-assumption}
 The graph~$\mathcal{G}$ is connected, and there is at least one node in~$\mathcal{G}$ with an anchor measurement.
\end{assumption}
This assumption is naturally fulfilled: if the network is supporting
several disconnected components, then each can be treated as a
different network, and, for disambiguation, range-based, static
localization requires the availability of 3 anchors in 2D and 4
anchors in 3D.
The convergence of Algorithm~\ref{alg:asyncronous} is stated next.
\begin{theorem}[Almost sure convergence]
  \label{th:asyn-convergence}
  Let Assumption~\ref{th:connected-assumption} hold. Consider
  Problem~\eqref{eq:cvx-matrix-z}, and the
  sequence~$\{z^{t}\}_{t \in \naturals}$ generated by
  Algorithm~\ref{alg:asyncronous}. Define the solution set as
  $\mathcal{Z}^{\star} = \{z \in \mathcal{Z}: F(z) =
  F^{\star}\}$. Then
  \begin{enumerate}
  \item $\mathrm{d}_{\mathcal{Z}^{\star}}(z^{t}) \to 0$, a.s.;
  \item $F(z^{t}) \to F^{\star}$, a.s.
  \end{enumerate}
\end{theorem}
We stress that the convergence stated in this result is the strongest
possible, encompassing convergence in expectation, \ie, the expected
value of iterates converge to a minimizer, and convergence in
probability, where the probability density function of the iterates
converges to the true probability density function of a
minimizer. Moreover, we prove that the iterates of our algorithm
converge to a minimizer of our function with probability one.

Next, we will state in Theorem~\ref{th:nr-iter-convergence} that, with
probability one, Algorithm~\ref{alg:asyncronous} converges in a
\emph{finite} number of iterations, for a given precision. Both proofs
can be found in~\ref{sec:proof-theor-refth}.
\begin{theorem}
\label{th:nr-iter-convergence}
For a prescribed precision~$\epsilon$, the sequence of
iterates~$\{z^{t}\}_{t \in \naturals}$ converges in~$K_{\epsilon}$
iterations. The expected count is
  \begin{equation}
    \label{eq:nr-iter-convergence}
    \expect \left[ K_{\epsilon} \right] \leq \frac{F\left(z^{0}\right)
    - F^{\star}}{b_{\epsilon}},
  \end{equation}
where~$b_{\epsilon}$ is a constant that depends on the
specified~$\epsilon$.
\end{theorem}
In practice, this result states that the number of iterations until
convergence will be a finite number, and not infinity, \eg, as in
typical distributed consensus algorithms.

\section{Numerical experiments}
\label{sec:numer-exper}

The experimental setup consists in a uniquely localizable geometric
network deployed in a square area with side of~1~Km, with four anchors
located at the corners, and ten sensors. The average node degree of
the network is~4.3.  The regular noisy range measurements are
generated according to
$ d_{ij} = | \|x_{i}^{\star} - x_{j}^{\star}\| + \nu_{ij} |$, and
$ r_{ik} = | \|x_{i}^{\star} - a_{k}\| + \nu_{ik} |,$
where~$x_{i}^{\star}$ is the true position of node~$i$,
and~$\{\nu_{ij} : i \sim j \in \mathcal{E}\} \cup \{\nu_{ik} : i \in
\mathcal{V}, k \in \mathcal{A}_{i}\}$ are independent Gaussian random
variables with zero mean and standard deviation~40~m.  Node~7 is
malfunctioning {\color{blue} and with some probability all
  measurements related to it are corrupted by noise with heavy-tailed
  distributions with a scale parameter of~4~Km, unless stated
  otherwise. Two families of heavy-tailed distributions were used:
  Laplace and Cauchy. The Laplace distribution is well-known
  for modeling outlier noise~\cite{Tarantola2005}, and Cauchy noise is
  frequent in contexts like radar echo, atmospheric noise, and
  underwater acoustic signals noise~\cite{tsakalides1998deviation}.}

Node~8 is badly calibrated, so measurements are 20\% of the true
value, which is also affected by white noise with the same standard
deviation of~40~m. We sampled both regular and outlier noise.  The
performance metric used to assess accuracy is the positioning error
per sensor, defined as
\begin{equation}
  \label{eq:error}
  \epsilon(m) =  \frac{\|\hat x(m) - x^{\star}\|}{|\mathcal{V}|},
\end{equation}
where~$\hat x(m)$ denotes the position estimates for all sensors in
Monte Carlo trial~$m$. The average positioning error is defined as
\begin{equation*}
  \label{eq:avgerror}
  \epsilon = \frac1{M} \sum_{m =1}^{M} \epsilon(m),
\end{equation*}
where~$M$ is the number of Monte Carlo trials.

\subsection{Underestimator performance}
\label{sec:under-perf}
The convex optimization problems were solved with
\texttt{cvx}~\cite{cvx}. We ran~200 Monte Carlo trials, where we sampled the
measurement and outlier noise sources.
\begin{figure}[tb]
  \centering
  {\includegraphics[width=0.8\textwidth]{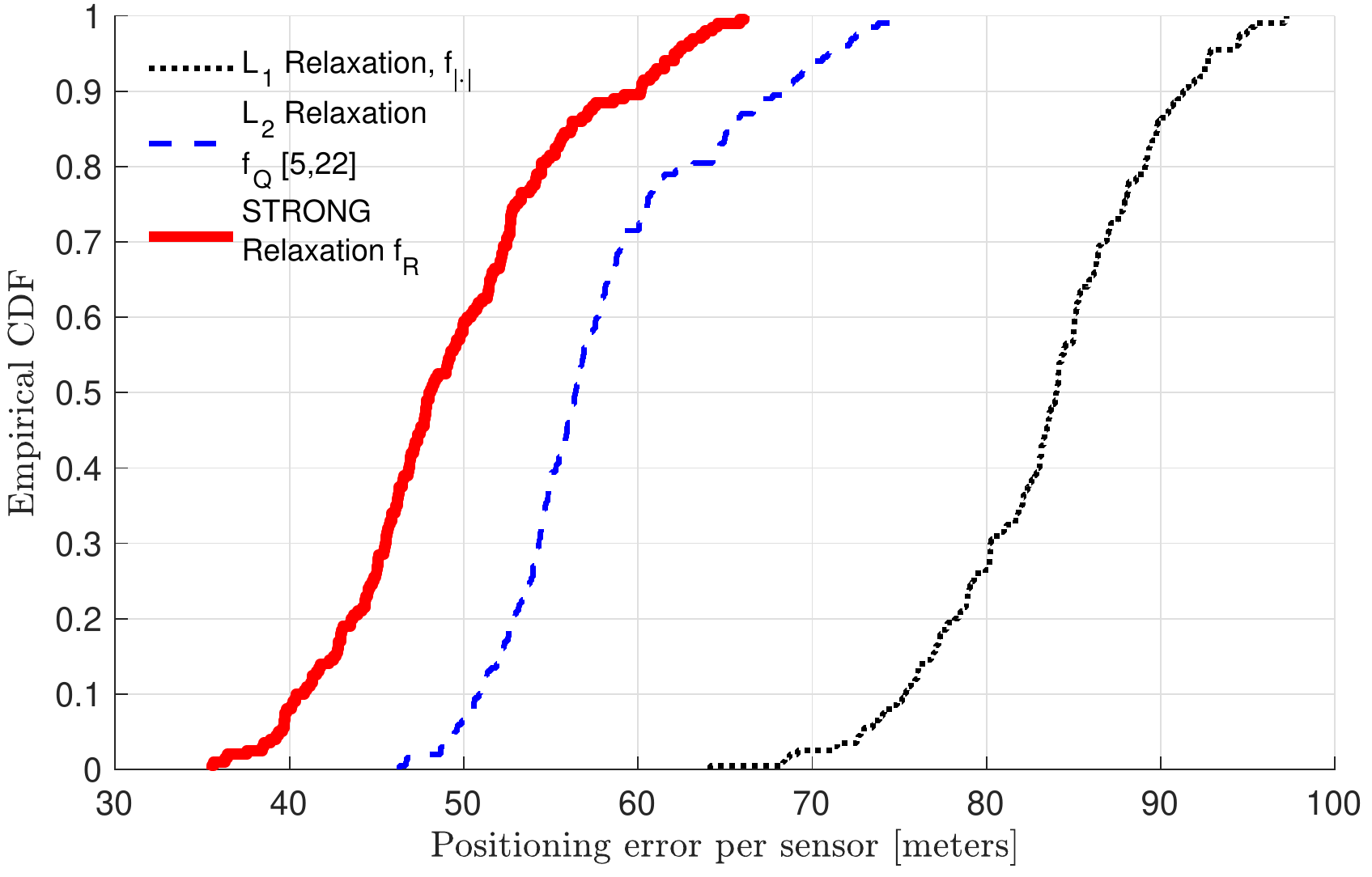}}
  \caption{Underestimator performance: Empirical CDF for the
    positioning error per sensor, in meters. {\color{blue} Our Huber-based
    relaxation shows higher accuracy than the $L_2$ and the robust
    $L_1$ relaxations, in an experiment where~2 nodes in~10 are
    affected by Laplacian noise.}} 
    \label{fig:error}
\end{figure}
The empirical CDFs of estimation errors are shown in
Figure~\ref{fig:error}, which demonstrate that the Huber robust cost
used in STRONG can reduce the error per sensor by more than~30
meters, when compared with the~$L_{1}$ discrepancy, and about 10~m for
the~$L_2$ discrepancy.
The sensitivity to the value of the Huber parameter~$R$
in~\eqref{eq:huber-loss} is {\color{blue} small in the range from 40
  to 100 m. In fact, for this range, the positioning error variation
  per sensor is about 10 cm. The proposed estimator
exhibits smaller error per sensor than the competing algorithms for all tested
values of the parameter.} We observe that the error increases when~$R$
approaches the standard deviation of the regular Gaussian noise,
meaning that the Huber loss gets closer to the~$L_{1}$ loss and, thus,
is no longer adapted to regular noise ($R=0$ corresponds exactly to
the~$L_{1}$ loss); in the same way, as~$R$ increases, so does the
extent of the quadratic section, and the estimator gets less robust to
outliers, so, again, the error increases.

\subsection{Performance of the distributed synchronous
  Algorithm~\ref{alg:synchronous}}
\label{sec:perf-distr-algor}

We tested Algorithm~\ref{alg:synchronous} using the same setup as in
the previous section, running~200 Monte Carlo trials.
\begin{figure}[tb]
  \centering
  {\includegraphics[width=0.8\columnwidth]{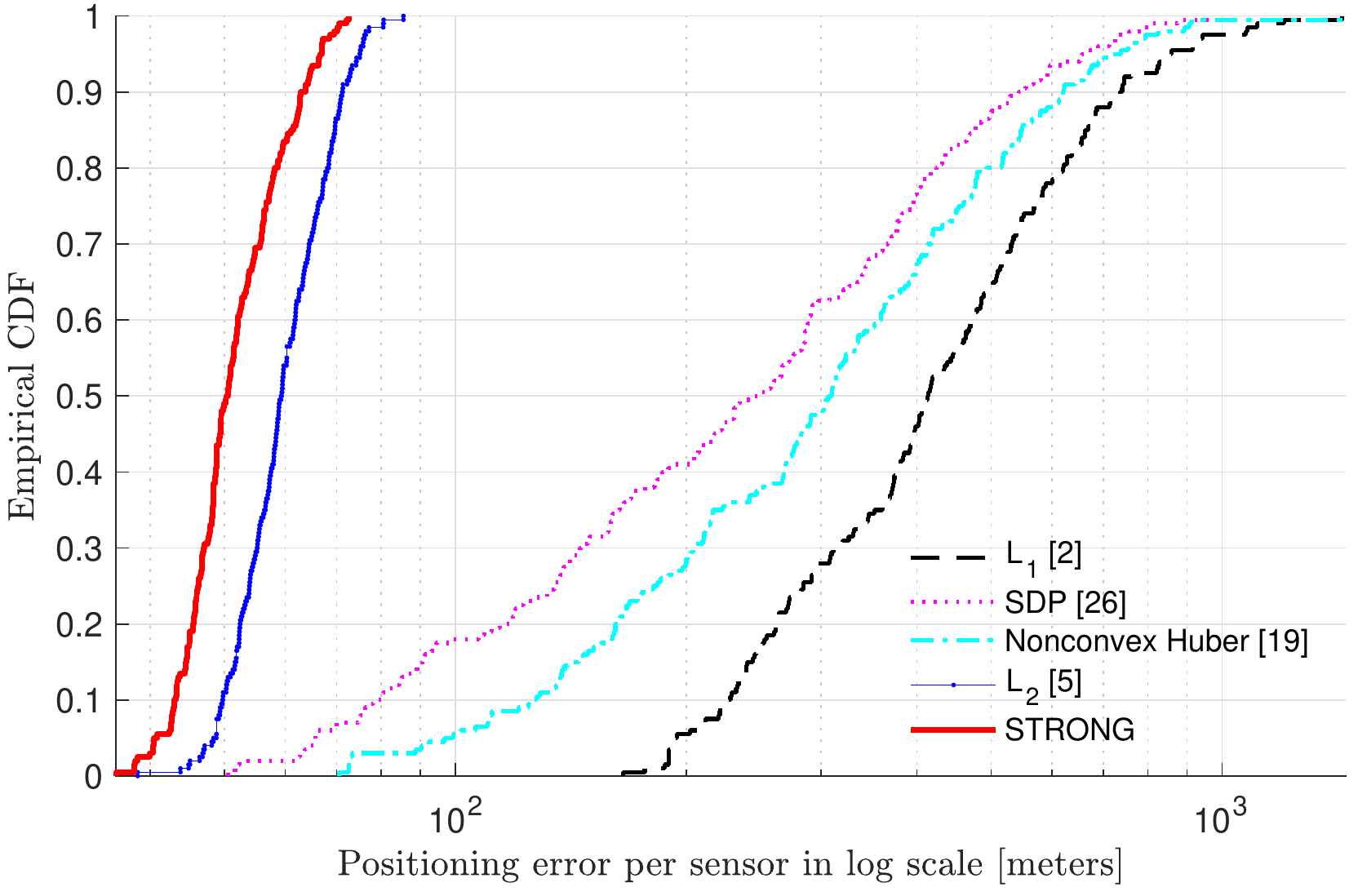}
    \caption{Accuracy of the distributed synchronous
      algorithm. Empirical CDF for the positioning error per sensor,
      in meters, for the outlier experiment,{\color{blue} with
        Laplacian noise. Our STRONG algorithm outperforms all tested
        methods,} including the nonconvex
      Huber~\cite{KorkmazVeen2009}, and the
      $L_2$~\cite{soares2014simple}. Note that although~$L_2$ seems to
      have a close performance at the scale of the plot, the accuracy
      gain of STRONG is still large.} 
  \label{fig:cdf-distr-median-SDP-L2-H-Laplace-noise}}
\end{figure}
\begin{figure}[tb]
  \centering
  {\includegraphics[width=0.8\columnwidth]{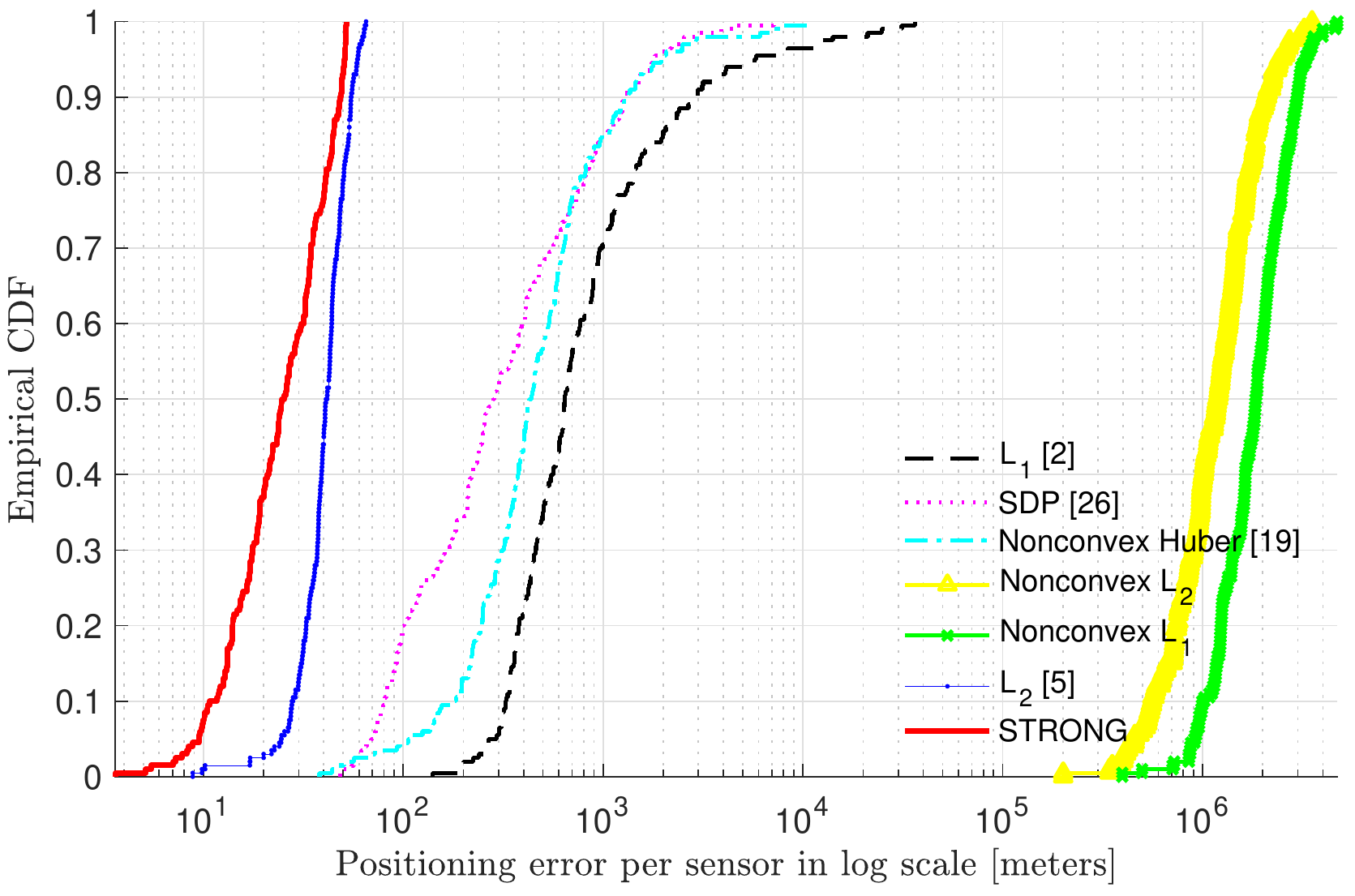}
    \caption{{\color{blue} Accuracy of the distributed synchronous
      algorithm. Empirical CDF for the positioning error per sensor,
      in meters, for the outlier experiment, with Cauchy noise. Our STRONG algorithm
      still outperforms all other methods.}} 
  \label{fig:cdf-distr-median-SDP-L2-H-Cauchy-noise}}
\end{figure}
We benchmark our method comparing with the performance of the
centralized solutions in O\u{g}uz-Ekim
\ea,~\cite{OguzGomesXavierOliveira2011}, the SDP presented by
Simonetto and Leus\footnote{In~\cite{SimonettoLeus2014}, the authors
  present a distributed ESDP algorithm which is a relaxation of the
  centralized SDP. As the simulation time for the distributed,
  edge-based algorithm is considerable we benchmarked against the
  tighter and more accurate centralized SDP
  solution.}~\cite{SimonettoLeus2014}, and also the distributed
locally convergent algorithm by Korkmaz and Van der Veen\footnote{This
  distributed method directly minimizes the nonconvex
  cost~\eqref{eq:huberDist}, thus delivering a local solution, that
  depends on the initialization point. The algorithm was initialized
  with samples of a zero mean Gaussian
  distribution.}~\cite{KorkmazVeen2009}. {\color{blue}The results are
  summarized in
  Figures~\ref{fig:cdf-distr-median-SDP-L2-H-Laplace-noise}, for
  Laplacian noise,
  and~\ref{fig:cdf-distr-median-SDP-L2-H-Cauchy-noise} for Cauchy
  noise.}  The empirical CDFs for positioning error~\eqref{eq:error}
{\color{blue} indicate that the STRONG algorithm has the highest
  accuracy among the tested methods.} {\color{blue}As depicted in
  Figures~\ref{fig:cdf-distr-median-SDP-L2-H-Laplace-noise} for
  Laplace and~\ref{fig:cdf-distr-median-SDP-L2-H-Cauchy-noise} for
  Cauchy noise,} it is noticeable that the STRONG algorithm
outperforms in accuracy the {algorithm in}~\cite{soares2014simple} for
the quadratic discrepancy by more than 14~meters per sensor in average
positioning error.  When we compare to a~$L_{1}$-type algorithm --- in
this case the~$L_1$ relaxation from O\u{g}uz-Ekim
\ea~\cite{OguzGomesXavierOliveira2011} --- under the same conditions,
the improvement of performance of our solution is, on average, about 120
meters per sensor.
\begin{figure}[tb]
  \centering
  {\includegraphics[width=0.8\columnwidth]{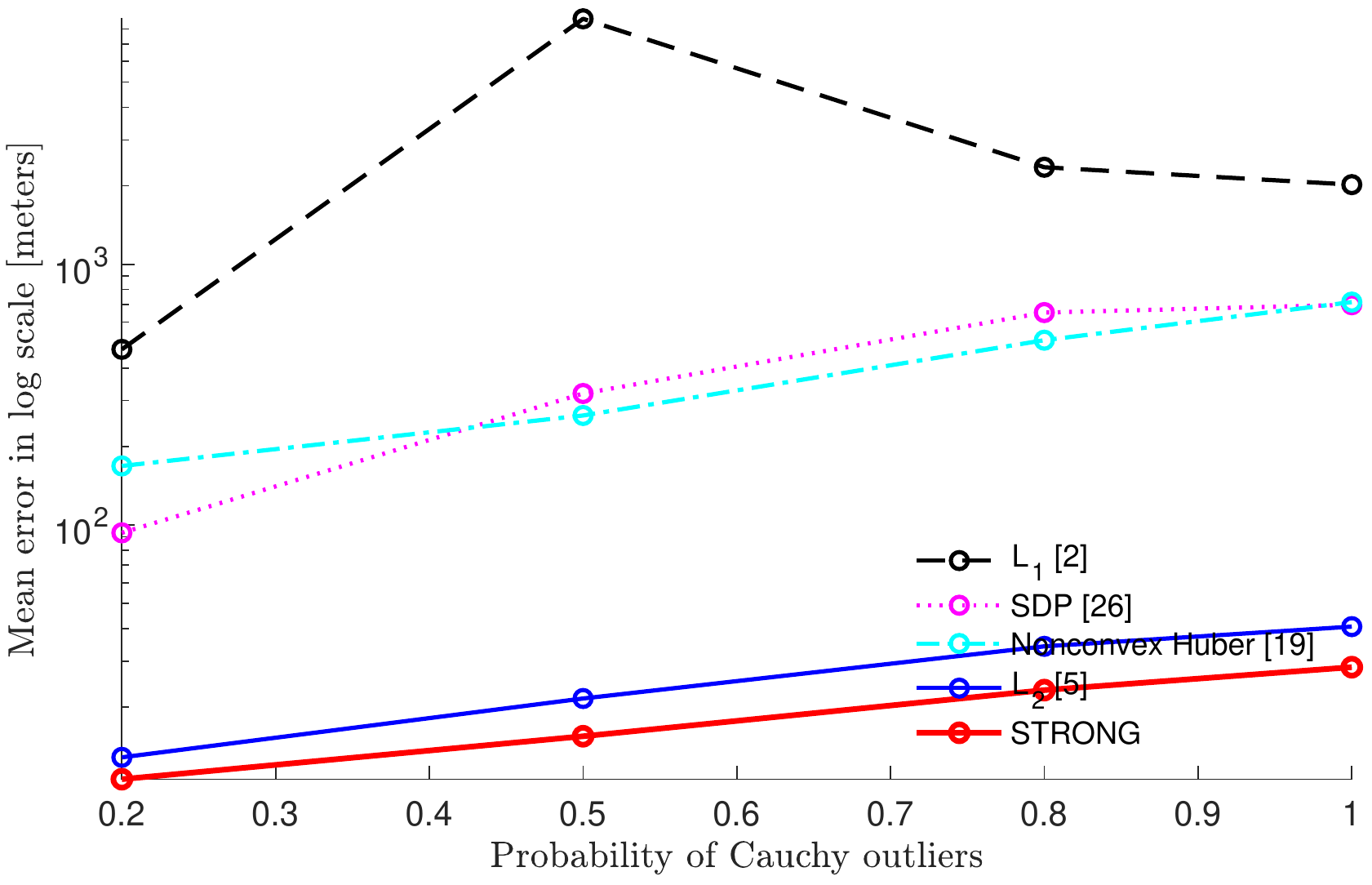}
    \caption{{\color{blue} As the probability of having one
        malfunctioning node injecting Cauchy outlier measurements into
        the distributed estimation process increases, all the methods increase
        the positioning error, as expected. The proposed STRONG method
        has the lowest mean error for 200 random trials for each value of
        the outlier probability.}} 
  \label{fig:median-SDP-L2-H-Cauchy-noise-meanErrVSprob}}
\end{figure}
\begin{figure}[tb]
  \centering
  {\includegraphics[width=0.8\columnwidth]{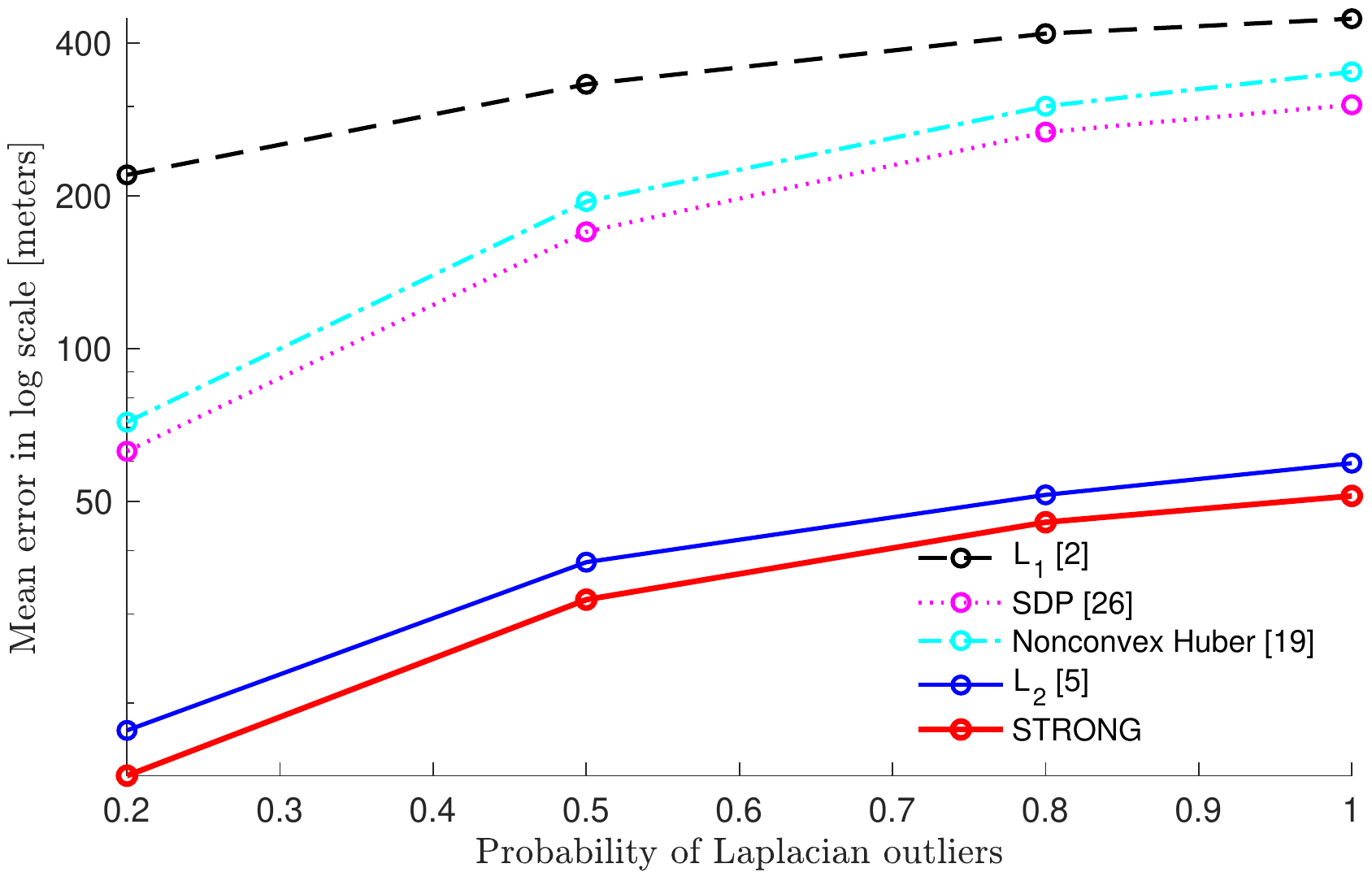}
    \caption{{\color{blue} As with Cauchy noise in
        Figure~\ref{fig:median-SDP-L2-H-Cauchy-noise-meanErrVSprob},
        The error response of the tested methods increases with the
        probability of multiplicative and Laplace noise.}}
  \label{fig:median-SDP-L2-H-Laplace-noise-meanErrVSprob}}
\end{figure}
{\color{blue} We also tested all methods regarding the probability of
  having a malfunctioning (or malicious) node in the network,
  injecting both multiplicative and Laplacian noise in
  Figure~\ref{fig:median-SDP-L2-H-Laplace-noise-meanErrVSprob} and
  Cauchy noise in Figure~\ref{fig:median-SDP-L2-H-Cauchy-noise-meanErrVSprob} into
  the estimation process. From the positioning error obtained in our
  simulations we see that outlier probability for the two heavy-tailed
  noise schemes drives the performance of the algorithms. We can also
  conclude that, for the tested scenarios, STRONG maintains
  consistently the lowest mean positioning error.}

{\color {blue} We have experimented with the Gaussian measurement
  noise power, but, in the presence of outliers and for the same
  outlier probability, the influence of varying the power of the
  measurement noise is negligible, compared with the effect of outlier noise.}

\subsection{Performance of the distributed asynchronous
  Algorithm~\ref{alg:asyncronous}}
\label{sec:perf-asynchr-algor}

Here, we tested Algorithm~\ref{alg:asyncronous}, async STRONG, using
Gaussian outlier noise with a standard deviation of~5~Km. We
benchmarked it against the synchronous
Algorithm~\ref{alg:synchronous}, STRONG, since both minimize the same
cost function. The algorithms were allowed to run with the same
communication load, and the mean positioning error for the considered
measurement noise levels (excluding outlier noise) is depicted in
Figure~\ref{fig:syncVSasync}.
\begin{figure}[tb] 
  \centering
  \includegraphics[width=0.8\columnwidth]{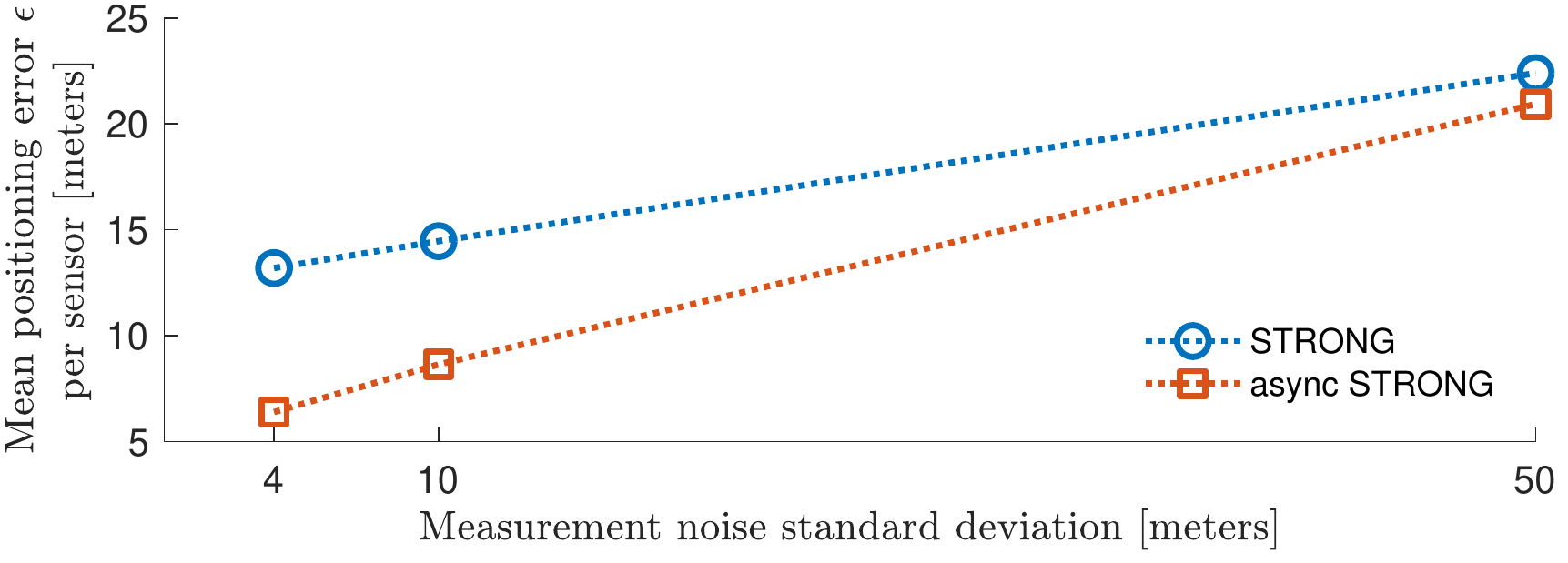}
  \caption{Mean positioning error for different noise standard
    deviations.}
  \label{fig:syncVSasync}
\end{figure}
\begin{figure}[tb] 
  \centering
  \includegraphics[width=0.8\columnwidth]{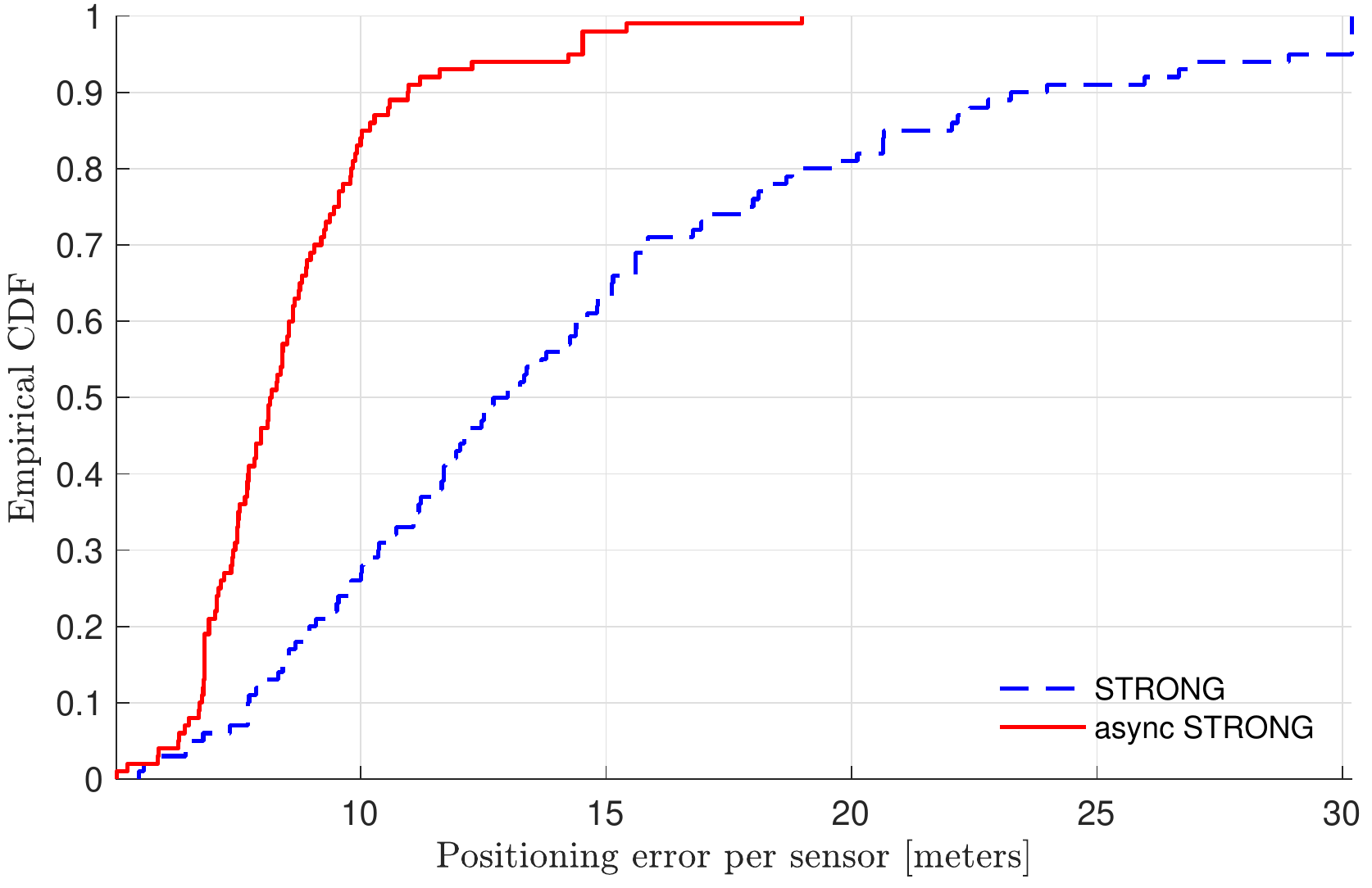}
  \caption{Accuracy of the distributed asynchronous algorithm:
    positioning error of synchronous algorithm~\ref{alg:synchronous}
    versus asynchronous algorithm~\ref{alg:asyncronous}. CDF of the
    positioning error. Noise standard deviation of 10 meters
    (corresponding to the second data point in
    Fig.~\ref{fig:syncVSasync}). Both algorithms were run with the
    same communication load. Experiment within a square with 1 Km
    sides.}
  \label{fig:cdf-sync-vs-async-sd2}
\end{figure}
We can observe that the async STRONG algorithm fares better than
STRONG for the same level of communication. This is an interesting
phenomenon that has already been empirically observed in several
optimization algorithms when comparing deterministic and randomized
versions. In fact, Bertsekas and
Tsitsiklis~\cite[Section~6.3.5]{BertsekasTsitsiklis1989} provide a
proof of this behavior for a restricted class of algorithms.
Figure~\ref{fig:cdf-sync-vs-async-sd2} further explores the numerical
results, by examining the CDF of the positioning error for the tested
Monte Carlo trials. Here we see the superior accuracy of the
asynchronous Huber Algorithm~\ref{alg:asyncronous}, for the same
communications volume.  We must, nevertheless, emphasize that this
result does not correspond to a faster algorithm, in terms of running
time: STRONG in one iteration updates all of the node positions in
parallel, and broadcasts the current estimates across neighbors,
whereas in async STRONG only one node operates at a time. As the
wireless medium might be much more intensively used for synchronous
updates than for random gossip interactions, it seems entirely
possible that for the same operation time STRONG will outperform async
STRONG --- at the expense of greater overall power consumption.

\section{Discussion and conclusions}
\label{sec:concl-future-work}

We presented two distributed, fast, and robust localization algorithms
that take noisy ranges and a few anchor locations, and output accurate
estimates of the node positions. We approximated the difficult
nonconvex problem based on the Huber discrepancy
in~\eqref{eq:snlOptProb} with a convex envelope of terms, robust to
outliers. How does the
Huber-based approximation in~\eqref{eq:huber-cvx} compare to
similar $L_{1}$ and $L_{2}$ underestimators, frequently used in robust
estimation?  A smaller optimality gap means a more robust
approximation~\cite{DestinoAbreu2011}: We designed a bound that
certifies the gap between the nonconvex and surrogate optimal values
for Huber, $L_{1}$ and $L_{2}$, which shows a tighter gap in the Huber
case. A numerical analysis of a star network in 1D unveiled that the
optimality gap for the Huber approximation was one order of magnitude
smaller than the quadratic or absolute value convexified problems, with
respect to their nonconvex counterparts. Numerical network
localization trials verify the robust behavior of this surrogate under
different types of outlier noise.
In order to develop a distributed method we needed to transform our
cost. So, we proposed a new representation of the Huber function
composed with a norm, and arrived at a novel distributed gradient
method, STRONG, with optimal convergence rate. But our STRONG
algorithm requires synchronization of nodes, which may be inconvenient in
many applications. Thus, we put forward a novel asynchronous method
for robust network localization, async STRONG, which converges with
probability one. Like other relaxation methods,
ours are prone to the anchor convex hull problem: preliminary results
show that the positioning accuracy degrades --- albeit graciously --- when
node positions depart from the convex hull of the anchors. Arguably, this is
not a big issue because engineers in general can control the choice or
placement of anchors, and can delimit the area under
survey.

In summary, both our algorithms work with simple computations at each
node and minimal communication requirements, have provable
convergence, and show superior performance in our numerical
experiments.

\section*{Acknowledgment}
The authors would like to thank Prof.\ Pinar O\u{g}uz-Ekim and Dr.\ Andrea
Simonetto for providing the MATLAB implementations of their published
algorithms. Also, we thank Prof.\ Jo\~{a}o Xavier for insightful
discussions during this research.

This research was partially supported by Fundação para a Ciência e a
Tecnologia (projects HARMONY PTDC/EEI-AUT/31411/2017, and
LARSyS - FCT Plurianual funding 2020-2023).

\appendix

\section{Proofs of Theorems~\ref{th:asyn-convergence} and~\ref{th:nr-iter-convergence}}
\label{sec:proof-theor-refth}
\subsection{Definitions}
\label{sec:definitions}
First, we review the definition of a block optimal point and describe
some useful mathematical objects used in the proofs.  

\begin{definition}
\label{def:block-optimal}
A point $z^{\bullet} = (z^{\bullet}_{i})_{i \in \mathcal{V}}$ is
\emph{block optimal} for the constrained function $F$
in~\eqref{eq:matrix-cost} if, for all $i$, $z^{\bullet}_{i} = (x^{\bullet}_{i},\{y^{\bullet}_{ij}\},\{w^{\bullet}_{ik}\})$ is a minimizer of the restricted function $F_{i}$ defined in \eqref{eq:sl-problem}~\cite{JakoveticXavierMoura2011}.
\end{definition}

We denote the sublevel sets of $F$ as $T_{\alpha} = \{z : F(z) \leq \alpha \}$ and the minimum of that function as $F^{\star}$. We also define the sets
\begin{eqnarray}
  \mathcal{Z}^{\star} &=& T_{F^{\star}} \cap \mathcal{Z}\\
  \label{eq:X-epsilon}
  \mathcal{Z}_{\epsilon} &=& \left \{z \in \mathcal{Z} : \mathrm{d}_{\mathcal{Z}^{\star}}(z) < \epsilon\right \}\\
  \label{eq:x-epsilon-c}
  \mathcal{Z}_{\epsilon}^{c} &=& \left \{z \in \mathcal{Z}: \mathrm{d}_{\mathcal{Z}^{\star}}(z) \geq \epsilon\right \}\\
  \label{eq:X-epsilon-c-f}
  \mathcal{\hat Z}_{\epsilon}^{c} &=& \mathcal{Z}_{\epsilon}^{c} \cap T_{F(z^{0})},
\end{eqnarray}
where $\mathcal{\hat Z}_{\epsilon}^{c}$ is the set of all points in $\mathcal{Z}$ whose distance
to the optimal set $\mathcal{Z}^{\star}$ is larger than $\epsilon$, but also
belong to the initial sublevel set of $F$. We will see that the iterates
of Algorithm~\ref{alg:asyncronous} belong to~$\mathcal{\hat Z}_{\epsilon}^{c}$
until they reach the absorbing set~$\mathcal{Z}_{\epsilon}$.
We also define the \emph{expected improvement function} as
\begin{equation}
  \label{eq:rel-improvement}
  \rho(z) = \expect \left [ F\left(Z(k+1) \right) | Z(k) = z \right] - F(z).
\end{equation}
It is easy to see that the expected improvement $\rho$ can also be
written as
\begin{equation}
  \label{eq:rel-improvement-2}
  \rho(z) = \sum_{i=1}^{n} \left( F_{i}^{\star}(z) - F(z) \right) P_{i},
\end{equation}
where $F_{i}^{\star}(z)$ denotes the minimum of $F_{i}$ in \eqref{eq:sl-problem} over its restricted argument for the current iterate $z$, and $P_{i}$ is the probability of the event ``node~$i$ is awaken at
time~$t$'' (we recall the independence of the random
variables~$\chi_{t}$ defined in~\eqref{eq:rv}). For notational
convenience, we introduce the function
\begin{equation}
  \label{eq:phi}
  \varphi(z) = - \rho(z),
\end{equation}
which, by construction of Algorithm~\ref{alg:asyncronous}, is always
non-negative.

\subsubsection{Auxiliary Lemmas}
\label{sec:auxil-lemmas}

The analysis is founded on
Lemma~\ref{lem:expected-improvement-properties}, where the symmetric
of the expected improvement is shown to attain a positive infimum on
the set~$\mathcal{\hat Z}_{\epsilon}^{c}$. Lemma~\ref{lem:basic-properties} will
be instrumental in the proof of
Lemma~\ref{lem:expected-improvement-properties} but contains also some
useful properties of function~$F$ and the solution
set~$\mathcal{Z}^{\star}$.

\begin{lemma}[Basic properties]
  \label{lem:basic-properties}
  Let $F$ be defined as in~\eqref{eq:matrix-cost}. Then the following
  properties hold.
  \begin{enumerate}
  \item $F$ is coercive;
  \item $F^{\star} \geq 0$ and $\mathcal{Z}^{\star} \neq
    \varnothing$;
  \item $\mathcal{Z}^{\star}$ is compact;
  \item If $z^{\bullet}$ is block optimal for $F$ in $\mathcal{Z}$, then it is
    a global minimizer for $F$ in $\mathcal{Z}$.
  \end{enumerate}
\end{lemma}
\begin{proof}
  \begin{enumerate}[leftmargin=*,noitemsep,topsep=0pt,parsep=0pt,partopsep=0pt]
  \item By Assumption~\ref{th:connected-assumption} in Section \ref{sec:analysis-asynchr-algor} there is a path
    from each node~$i$ to some node~$j$ which is connected to an
    anchor~$k$. Also, we know that, by definition,~$\|y_{ij}\|$
    and~$\|w_{ik}\|$ are bounded by the ranges~$d_{ij} < \infty$
    and~$r_{ik} < \infty$. So these components of~$z$ will have no
    effect in the limiting behavior of~$F$. If $\|x_{i}\| \to \infty$
    there are two cases: (1) there is at least one edge $t \sim u$
    along the path from~$i$ to~$j$ where~$\|x_{t}\| \to \infty$
    and~$\|x_{u}\|\not \to \infty$, and so
    $h_{R_{tu}}(\|x_{t}-x_{u} - y_{tu}\|) \to \infty$; (2) if
    $\|x_{u}\| \to \infty$ for all~$u$ in the path between~$i$
    and~$j$, in particular we have~$\|x_{j}\| \to \infty$ and
    so~$h_{Ra_{jk}}(\|x_{j}-a_{k} - w_{jk}\|) \to \infty$, and
    in both cases~$F \to \infty$. Thus,~$F$ is coercive.
  \item $F$ defined in~\eqref{eq:matrix-cost} is a continuous, convex
    and real valued function lower bounded by zero; so, the
    infimum~$F^{\star}$ exists and is non-negative. To prove that this
    infimum is attained and~$\mathcal{Z}^{\star} \neq \varnothing$, we
    observe that the set~$\mathcal{Z}$ is a cartesian product of the
    closed
    sets~$\reals^{np}$,~$\{y \in \reals^{|{\mathcal E}| p} : y =
    \{y_{ij}, \|y_{ij}\| \leq d_{ij}, i \sim j \in \mathcal{E}\} \}$,
    and~$\{w \in \reals^{\sum_i |{\mathcal A}_i| p} : w = \{w_{ik},
    \|w_{ik}\| \leq r_{ik}, i \in \mathcal{V}, k \in \mathcal{A}_i\}
    \}$, and thus~$\mathcal{Z}$ is also closed. Now consider a sublevel
    set~$T_{\alpha}$, which is compact because $F$ is coercive. For
    some~$\alpha$, the intersection of~$T_{\alpha}$ and~$\mathcal{Z}$
    is nonempty and it is known that the intersection of a closed and
    a compact set is compact~\cite[Corollary to
    2.35]{rudin1976principles}, so~$T_{\alpha}\cap\mathcal{Z}$ is
    compact. As function~$F$ is convex, it is also continuous on the
    compact set~$T_{\alpha}\cap\mathcal{Z}$ and, by the extreme value
    theorem, the
    value~$p = \inf_{z \in T_{\alpha}\cap\mathcal{Z}} F(z)$ is
    attained. It is obvious
    that~$\inf_{z \in T_{\alpha}\cap\mathcal{Z}} F(z) = \inf_{z \in
      \mathcal{Z}} F(z)$.
  \item $\mathcal{Z}^{\star}
    = T_{\alpha} \cap \mathcal{Z}$ for~$\alpha =
    F^{\star}$, and we established in the previous proof that~$ T_{\alpha}
    \cap \mathcal{Z}$ is compact.
  \item If~$z^{\bullet}$
    is block-optimal, then~$\langle
    \nabla F_i(z^\bullet_i), z_i-z_i^\bullet\rangle \geq
    0$ for all~$i$.
    When stacking the inequalities for all~$i$,
    we get~$\langle
    \nabla{F}(z^\bullet),z-z^\bullet\rangle \geq
    0$, which proves the claim.\qedhere
  \end{enumerate}
\end{proof}

\begin{lemma}
\label{lem:expected-improvement-properties}
  Let~$\varphi$ be defined as in~\eqref{eq:phi}, taking values on the
  set~$\mathcal{\hat Z}_{\epsilon}^{c}$ in~\eqref{eq:X-epsilon-c-f}. Then,
  \begin{enumerate}
  \item Function $\varphi$ is positive:
    \begin{equation}
      \label{eq:expected-improvement-property-1}
      \varphi(z) > 0, \quad \text{for all }  z \in \mathcal{\hat Z}_{\epsilon}^{c};
    \end{equation}
  \item As a consequence, function~$\varphi$ is bounded below by a finite
    positive value~$a_{\epsilon}$:
    \begin{equation}
      \label{eq:expected-improvement-property-2}
      \inf_{z \in \mathcal{\hat Z}_{\epsilon}^{c}} \varphi(z) = a_{\epsilon}.
    \end{equation}
  \end{enumerate}
\end{lemma}
\begin{proof}
  We start by proving the first claim, $\varphi(z) > 0$ for all $z \in
  \mathcal{\hat Z}_{\epsilon}^{c}$. Suppose $\varphi(z) =
  0$; then, by Equation~\eqref{eq:rel-improvement-2} $F^{i}(z) =
  F(z)$, which means that
  $z$ is block optimal. By Lemma~\ref{lem:basic-properties},
  $z$ is also global optimal, which contradicts the fact that
  $z$ belongs to the set~$\mathcal{\hat
    Z}_{\epsilon}^{c}$.  The second claim follows by observing that
  $\varphi$ is a sum of real valued functions and, thus, a real valued
  function, and that $\varphi(z)$ is bounded below by zero in~$
  \mathcal{\hat
    Z}_{\epsilon}^{c}$ and so it has a positive infimum for~$
  \mathcal{\hat Z}_{\epsilon}^{c}$.
\end{proof}

\subsubsection{Theorems}
\label{sec:theorems}
Equipped with the previous Lemmas, we are now ready to prove the
Theorems stated in Section~\ref{sec:analysis-asynchr-algor}.
\begin{proof}[Proof of Theorem~\ref{th:asyn-convergence}]
  We denote the random variable corresponding to the outcome of the
  $t$-th loop step of Algorithm~\ref{alg:asyncronous} as~$Z^{t}$. The expected
  value of the expected improvement function~$\rho$ \eqref{eq:rel-improvement} is
  \begin{equation*}
    \begin{split}
      \expect \left[ \rho\left(Z^{t}\right) \right] & = \expect \left[
        \expect \left[F \left(Z^{t+1} \right) | Z^{t} \right] \right]
      - \expect\left[ F\left(Z^{t}\right) \right] \\ & = \expect \left
        [F \left( Z^{t+1} \right) \right] - \expect \left[F \left(
          Z^{t} \right) \right],
    \end{split}
  \end{equation*}
  where the second equality comes from the tower property documented,
  \eg, in Williams~\cite{Williams1991}. This expectation can also be
  written as
\begin{equation*}
  \begin{split}
    \expect \left[ \rho\left(Z^{t}\right) \right] &= \expect \left[
      \rho \left(Z^{t} \right) | Z^{t} \in \mathcal{\hat
        Z}_{\epsilon}^{c} \right] \prob \left(Z^{t} \in \mathcal{\hat
        Z}_{\epsilon}^{c} \right) + \expect \left[ \rho \left(Z^{t}
      \right) | Z^{t} \not \in \mathcal{\hat Z}_{\epsilon}^{c} \right]
    \prob \left(Z^{t} \not \in \mathcal{\hat Z}_{\epsilon}^{c} \right)
    \\ & \leq \expect \left[ \rho \left(Z^{t} \right) | Z^{t} \in
      \mathcal{\hat Z}_{\epsilon}^{c} \right] \prob \left(Z^{t} \in
      \mathcal{\hat Z}_{\epsilon}^{c} \right).
  \end{split}
\end{equation*}
By combining both we get
\begin{equation*}
  \expect \left [F \left( Z^{t+1} \right) \right] - \expect
  \left[F \left( Z^{t} \right) \right] \leq \expect \left[ \rho
    \left(Z^{t} \right) | Z^{t} \in \mathcal{\hat Z}_{\epsilon}^{c} \right] \prob
  \left(Z^{t} \in \mathcal{\hat Z}_{\epsilon}^{c} \right)
\end{equation*}
which can be further bounded using
Lemma~\ref{lem:expected-improvement-properties} as
 \begin{equation*}
\expect \left [F \left( Z^{t+1} \right) \right] - \expect
\left[F \left( Z^{t} \right) \right] \leq -a_{\epsilon} p_{t}
\end{equation*}
where $p_{t} = \prob \left(Z^{t} \in \mathcal{\hat Z}_{\epsilon}^{c} \right).$
 By
expanding the recursion we obtain
\begin{equation*}
  \expect \left [F \left( Z^{t+1} \right) \right] \leq
  -a_{\epsilon}\sum_{k=1}^{t}p_{k} + F\left(z^{0}\right),
\end{equation*}
which provides a bound on the sum of probabilities~$p_{k}$ when
rearranged as
\begin{equation*}
  \sum_{k=1}^{t}p_{k} \leq \frac{F \left( z^{0} \right) - \expect \left [F \left( Z^{t+1} \right) \right]}{a_{\epsilon}}.
\end{equation*}
Taking $t$ up to infinity, we obtain
\begin{equation*}
  \begin{split}
    \sum_{k=1}^\infty p_k &\leq \frac{F \left( z^{0} \right) - \expect
      \left [F \left( Z^{\infty} \right) \right]}{a_\epsilon} \\ & \leq
    \frac{F \left( z^{0} \right) - F^\star}{a_\epsilon}.
  \end{split}
\end{equation*}
This means that the infinite series of probabilities $p_{k}$ takes on a
finite value; by the Borel-Cantelli Lemma, we get
  $\prob \left( Z^{t} \in \mathcal{\hat Z}_{\epsilon}^{c}, \quad i.o. \right) = 0,$
where~$i.o.$ stands for \emph{infinitely often}. This concludes the
proof, since this statement is equivalent to the first claim of
Theorem~\ref{th:asyn-convergence}.
\end{proof}

\begin{proof}[Proof of Theorem~\ref{th:nr-iter-convergence}]
  We consider a variant of the set defined in~\eqref{eq:X-epsilon}
  \begin{equation*}
    \mathcal{Y}_\epsilon =T_{F^\star + \epsilon} \cap \mathcal{Z},
  \end{equation*}
  and, similarly to \eqref{eq:x-epsilon-c}--\eqref{eq:X-epsilon-c-f}, its complement in $\mathcal{Z}$, $\mathcal{Y}_\epsilon^c$,
  and the intersection with the initial sublevel set, $\mathcal{\hat{Y}}_\epsilon^c = \mathcal{Y}_\epsilon^c \cap T_{F(z^{0})}$
  Using the same arguments as in
Lemma~\ref{lem:expected-improvement-properties}, we can prove that~$\inf_{z \in \mathcal{\hat{Y}}_{\epsilon}^{c}} \varphi (z) = b_{\epsilon}, \quad 0 < b_{\epsilon} < \infty.$
We now define a sequence of points $\tilde z ^{t}$ such that
\begin{equation*}
\tilde z^{t} = 
  \begin{cases}
    z^{t} & \text{if } z^{t} \in \mathcal{\hat Y}_{\epsilon}^{c} \\
    z^{\star} & \text{otherwise},
  \end{cases}
\end{equation*}  
and the sequence of real values~$\rho(\tilde z^{t}) =
  \begin{cases}
    \rho(z^{t})  & \text{if } z^{t} \in \mathcal{\hat Y}_{\epsilon}^{c} \\
    0 & \text{otherwise}
  \end{cases}.$ The expected value of $\rho(\tilde Z^{t})$
  is
  \begin{equation*}
    \expect \left[\rho \left(\tilde Z^{t} \right) \right] = \expect
  \left[ F \left(\tilde Z^{t+1} \right) \right] - \expect \left[ F
    \left(\tilde Z^{t} \right) \right].
  \end{equation*}
  Summing these expectations
  over time, we
  get
  \begin{equation*}
    \sum_{k=0}^{t-1} \expect \left[\rho \left(\tilde Z^{k} \right)
  \right] = \expect \left[ F \left( \tilde Z^{t} \right) \right] - F
  \left( z^{0} \right).
  \end{equation*}
  Taking $t$ to infinity and interchanging
  integration and summation we
  obtain
  \begin{equation*}
    \expect \left[ \sum_{k=0}^{\infty} \rho\left( \tilde Z^{k}
    \right) \right] = \expect \left[ F \left( Z^{\infty} \right)
  \right] - F\left( z^{0} \right).
  \end{equation*}
  From the definition
  of~$\rho(\tilde z^{k})$ we can
  write~$\expect \left[ \sum_{k=0}^{\infty} \rho\left( \tilde Z^{t}
    \right) \right] \leq \expect \left[ K_{\epsilon}(-b_{\epsilon})
  \right],$ 
thus obtaining the result
\begin{equation*}
  \expect \left[ K_\epsilon \right] \leq \frac{F (z^{0}) -
    \expect \left[ F \left (Z^{\infty}\right) \right]}{b_\epsilon}
  \leq \frac{F(z^{0}) - F^\star}{b_\epsilon}
\end{equation*}
which is a finite number. This completes the proof.
\end{proof}


\bibliographystyle{elsarticle-num}
\bibliography{biblos}

\end{document}